\documentclass[11pt]{amsart}
\usepackage{amssymb}
\usepackage{colonequals}
\usepackage{relsize}

\usepackage[mathscr]{euscript}
\usepackage{amsfonts}

\newtheorem*{theorem*}{Theorem}
\usepackage{amsmath}
\usepackage[usenames, dvipsnames]{xcolor}
\date{\today}
\def \B {{\mathfrak  B}}
\def \U {{\mathcal  U}}

\def\bel{\begin{equation}\label}
\def\eeq{\end{equation}}
\def\cS{\mathfrak{T}}

\def\h{\Psi}
\def \Ba {\mathbb{B}}
\def \v{{\bf v}}
\def \cbf{{\bf c}}
\def \va{\mathpzc{v}}

                                                                                                                                                                                                                                                                                                                                            \def\dime{\mathfrak{d}}

\def \vecs{\vec{\bf s}}
\def \vecc{\vec{\bf c}}
\def \vech{\vec{\bf h}}
\def \veceps{\vec{\bf \epsilon}}

\def\vz{\mathpzc{v}}

\def\vsm{\vskip0.3truecm}

\newtheorem{remark}{Remark}[section]
\newtheorem{definition}{Definition}[section]

\newtheorem{theorem}{Theorem}[section]
\newtheorem{lemma}[theorem]{Lemma}
\newtheorem{proposition}[theorem]{Proposition}
\newtheorem{corollary}[theorem]{Corollary}

\def\ds{\displaystyle}

\def\bega{\begin{array}}
\def\enda{\end{array}}

\def\bepmatrix{\begin{pmatrix}}
\def\enpmatrix{\end{pmatrix}}

\def\vsm{\vskip0.3truecm}

\def\bel{\begin{equation}\label}
\def\eeq{\end{equation}}
\newcommand\ee{\end{equation}}
\def\benl{\begin{equation*}}
\def\eenl{\end{equation*}}

\def\be{\begin{equation}}
\def\beq{\begin{equation}}
\def\bel{\begin{equation}\label}
\def\eeq{\end{equation}}

\newcommand\ba{\begin{array}}
\newcommand\ea{\end{array}}

\def\vs{\vskip 2em}

\def\begi{\begin{itemize}}
\def\endi{\end{itemize}}

\def\d{ {\rm d} }

\def\R{I\!\!R}

\def\R{I\!\!R}
\newcommand{\cR}{\mathbb{R}}

\newcommand{\cN}{\mathbb{N}}

\def\A{\mathcal{A}}
\def\B{\mathfrak{B}}
\def\C{\mathcal{C}}

\def\K{\mathcal{K}}

\def\R{\mathcal{R}}

\def\U{\mathcal{U}}

\def\ab{\bar{a}}

\def\sb{\bar{s}}

\def\ub{\bar{u}}

\def\xbo{\check x}

\def\wb{\bar{w}}
\def\xb{\bar{x}}
\def\yb{\bar{y}}

\def\eps{\varepsilon}

\def \dd{\mathrm{d}}

\def\sol{\textcolor{blue}}

\DeclareMathAlphabet{\mathpzc}{OT1}{pzc}{m}{it}

\title[Nonlinear impulsive systems]{ A Higher-order Maximum Principle \\ for Impulsive Optimal Control Problems}

\begin{document}

\vs
\author[M.S. Aronna]{M. Soledad Aronna}
\address{M.S. Aronna,  Escola de Matem\'atica Aplicada, Funda\c{c}\~ ao Get\'ulio Vargas,  Rio de Janeiro 22250-900, Brazil
 }
\email{soledad.aronna@fgv.br}

\author[M. Motta]{Monica Motta}
\address{M. Motta, Dipartimento di Matematica,
Universit\`a di Padova\\ Padova  35121, Italy}
\email{motta@math.unipd.it}

\author[F. Rampazzo]{Franco Rampazzo}
\address{F. Rampazzo, Dipartimento di Matematica,
Universit\`a di Padova\\ Padova  35121, Italy}
\email{rampazzo@math.unipd.it}

\maketitle
\begin{abstract}
 We consider a nonlinear system, affine with respect to an unbounded control $u$ which is allowed to range    in a closed  cone.  To this system we associate a  Bolza type minimum problem, with a Lagrangian  having sublinear growth with respect  to  $u$.  This lack of coercivity gives the problem an  {\it impulsive} character,  meaning  that  minimizing sequences of trajectories  happen to converge towards   discontinuous paths. As is known, a distributional approach does not make sense in such a nonlinear setting, where, instead, a suitable embedding in the graph-space is needed.  
 We provide higher order  necessary optimality conditions for  properly defined  impulsive minima, in the form of equalities and inequalities  involving iterated Lie brackets of the dynamical vector fields. These conditions are derived under very weak regularity assumptions and without any constant rank conditions.
\end{abstract}

\thanks

\section{Introduction}
In this paper we establish necessary optimality conditions for the {\it space-time, impulsive extension} of the free end-time  optimal control problem
  \be
\label{cost}
\text{minimize} \ \h(T, x(T)) +\int_0^T \ell(x(t),u(t),a(t))\,dt,
\ee
where the minimization is performed over the set of   {\it processes} $(T, u, a, x,\vz)$ verifying 
\bel{E}
\left\{
\begin{array}{l}
\ds \frac{dx}{dt}(t) = f(x(t), a(t)) + \ds\sum_{i=1}^m  g_i(x(t)) {u^i}(t), \quad \text{a.e. $t\in [0,T]$,} \\
\ds\frac{d\vz}{dt}(t) = |u(t)|,\\ [1.5ex]
\ds (x,\vz)(0)=({\xbo},0), \quad  \vz(T)\leq K,\quad (T, x(T)) \in \cS.
\end{array}\right.
\eeq
Here, $0<K\le+\infty$, the {\it target}  $\cS$  is a closed subset of $\cR_+\times \cR^{n}$,   the control $a$ ranges in a compact set  $A\subset\cR^q$, and standard regularity hypotheses are verified by the vector fields $f$, $g_i$ and the cost functions $\ell$, $\Psi$. Instead, 
less usual  assumptions  are  made on the  control maps $u$ and  on the $u$-growth of the  Lagrangian $\ell$.  Precisely: 

 \noindent
 \ (i) the  {\it unbounded} controls $u$, which 
 take values  in a closed cone ${\mathcal C}\subseteq\cR^m$, are $L^1$ functions verifying
 $
  \|u\|_1:=\int_0^T |u(t)| dt=\vz(T) \leq K;
$  

\noindent
\ (ii) the {\rm Lagrangian}  $\ell: \cR^n\times\C\times A \to \cR$ has the form $\ell(x,u,a)=\ell_0(x,a)+\ell_1(x,u), $ 
 with a continuous  {\it recession function} $\hat\ell_1$ given by
 $$
 \hat\ell_1(x,w^0,w):= \ds\lim_{r\to w^0 } r\ell_1\left(x,\frac{w}{r}\right),
 \quad \text{for } (x,w^0,w) \in \cR^n\times \cR_+\times \C.$$  In particular $\ell$ has sublinear growth in $u$.

On the  one hand, optimal control problems  with such a slow $u$-growth in the cost   are motivated by several applications \cite{BressanAldo1989,BressanRampazzo2010,GajRamRap08,Catlla2008,Marle1991,AzimovBishop2005, Rampazzo1991,BressanMotta1993}.  For instance, a dynamics  affine in the unbounded controls governs  the motion of a mechanical system of mutually linked rigid bodies. In that case, the controlled parameters are  the speeds of the  {\it shape } coordinates.

On the other hand,  the lack of  a sufficiently fast $u$-growth  of the cost $\ell$ 
may cause minimizing sequences of trajectories to tend towards discontinuous paths.\footnote{ For instance, if one considers the  minimum  time  problem --i.e. $\ell\equiv 1$ and $\Psi\equiv 0$-- with a  target in $\cR^n$  intersecting  the orbit of the vector field $g_1$ issuing from the initial point $\check x$,  the infimum  value is  zero and the  `extended' optimal trajectory should run through the mentioned orbit {\it with infinite speed}.} 
Motivated by that, and following a nowadays standard approach
 \footnote{Because of the nonlinearity of the dynamics, a {\it distributional} interpretation lacks essential prerequisites for robustness \cite{Haj85}.}  \cite{Warga1972,Rishel1965,BressanRampazzo1988,MottaRampazzo1995,
Millerbook,SilvaVinter1996,GuerraSarychev2015},
one `compactifies' the problem by embedding the original control system 
in the {\it extended,  space-time,} system
\bel{extendedintro}
  \left\{\begin{array}{l}
\ds\frac{d{y^0}}{ds}  (s) = w^0(s),  \\[0.1cm]
\ds\frac{dy}{ds} (s) = f( y(s),\alpha(s))w^0(s)+ \sum_{i=1}^{m}g_{i}( y(s))w^i(s), \\[0.1cm]
\ds\frac{d\beta}{ds} (s) = |w(s)|, \\ [0.2cm]
(y^0,y,\beta)(0)={(0,\check x,0)}, \\ (y^0(S), y(S),\beta(S)) \in \cS\times[0,K],
\end{array}\right.  \quad \text{a.e. $s\in[0,S]$,}
\eeq
and considering   the {\it extended} cost functional
  \be
\label{costext}
 \h(y^0(S), y(S) )  +\int_0^S \ell^e(y(s),w^0(s), w(s),\alpha(s))ds,
\ee
where
$
 \ell^e(x,w^0,w,a) := \ell_0(x,a)w^0 + \hat\ell_1(x,w^0,w).
$
In particular, for  problem \eqref{extendedintro}-\eqref{costext} one can consider {\it bounded} controls verifying  $w^0(s)+|w(s)|=1$, $w(s)\in \C$, $w^0(s)\geq 0$, and $a(s)\in A$, for a.e. $s\in[0,S]$. As soon as $w^0>0$ a.e.,  \eqref{extendedintro}-\eqref{costext} is nothing but a time reparameterization  
 of  the original control problem \eqref{cost}-\eqref{E}, and the unboundedness of $u$ is reflected in the possibility of taking $w^0$ going to $0$.
By allowing processes  $(S,w^0,w,\alpha, y^0,y,\beta)$ such that $w^0=0$ a.e.  on non-degenerate subintervals $[s_1,s_2]$, we are embedding  \eqref{cost}-\eqref{E}
 in a more general problem. More precisely, the  time-variable $t=y^0(s)$  has a constant value $\bar t$  on  such an interval  $[s_1,s_2]$, while 
   the space trajectory  $y(s)$   evolves according to  the 
nonlinear  dynamics $\frac{dy}{ds}=\sum_{i=1}^{m}g_{i}( y)w^i.$ In particular, the jump
$x(\bar t+)-x(\bar t-) = y(s_2)-y(s_1)$ depends on the restriction $w_{|_{ [s_1,s_2]}}$.
Our necessary  conditions concern a minimum for the extended, space-time, problem  \eqref{extendedintro}-\eqref{costext}.
We begin by  exploiting the {\it rate-independence} of the extended  problem  in order  to establish  a First Order Maximum Principle, in Theorem \ref{PMPe}.  While the idea is nothing but new --see, for instance, \cite{SilvaVinter1997,PereiraSilva2000,Millerbook,AruKarPer10,MottaRampazzoVinter2019}-- here  we stress  a  kind of competition among the  Hamiltonian corresponding to the drift and  the `non-drift'  Hamiltonian.
 
 However, the  main novelty of the present paper  consists  in a   Maximum Principle containing   {\it higher-order} necessary conditions which involve iterated Lie brackets (see Theorem \ref{PMPeho}). The crucial tool to prove our conditions consists in the construction of variations by  coupling asymptotic formulas for Lie brackets and suitable   insertions of  impulsive intervals. In particular,  in connection with  sets of lines contained in the cone $C$  (see Section \ref{SecProblem}), the dual product of the  adjoint path  with the corresponding iterated Lie brackets turns out to vanish.
 
  Lie bracket-involving  necessary  conditions have been widely investigated  within  classical  geometric control theory, a quite incomplete list of reference being  \cite{ABDL12,Kno80,Krener1977,Sussmann1999geometry,SchaettlerLedzewicz2012,barbero2012presymplectic}. 
Instead, as far as impulsive control theory is concerned, the  only 
 results involving  higher order --actually, second order-- conditions deal, at our knowledge,   with the so-called {\it  commutative case},  where $[g_i,g_j]\equiv0$ for all  $i,j=1,\ldots,m,$ (see e.g.  \cite{ArutyunovKaramzin2018,AruDykPer05,Dykhta1994}).   
  Let us  mention that  our conditions might be  regarded as a generalization to impulsive trajectories of   \cite{ChittaroStefani2016}, where one considers   
 the (non extended) minimum time problem with  $L^\infty$, but not a priori bounded, controls { (see Remark \ref{rcs}). }
Let us stress that we assume neither  the invariance    of  the dimension of the Lie algebra generated by  $g_1,\ldots,g_m$  nor the mere existence of such algebra. This is made possible by the fact that our Lie bracket-shaped  variations are based on rather general asymptotic formulas,  which demand  very low regularity on the vector fields: in particular,  most of these formulas assume {\it just the continuity} of the involved Lie brackets.
 \vsm
The article is organized as follows. In  Subsection \ref{SecNotation}  we introduce some general notation and  a set of definitions and technical results involving Lie brackets. The optimal control problem is described in detail  in Section \ref{SecProblem}, together with its {\em space-time} extension.  In Section \ref{SNC}, Theorem \ref{PMPe},  we state a  First Order Maximum Principle. In Section \ref{SMPho}, Theorem \ref{PMPeho},  
we establish our main result, namely the  Higher Order Maximum Principle, whose proof is given in  Section \ref{SecProof}.

 
\subsection{Notations and preliminaries}\label{SecNotation}
\subsubsection{Some basic notation}
Let $N\ge1$ be an integer. For any $i\in\{1,\dots,N\}$, we write ${\bf e}_i$  for the $i$th element of the  canonical basis of $\cR^N.$ Given $\check{x}\in \cR^N,$  
$\Ba_N({\check x}):=\{x\in \cR^N : |x-{\check x}| \leq 1\}$,  $\Ba_N:=\Ba_N(0)$,  and $\partial \Ba_N:=\{x\in \cR^N: \  |x|=1\}$.   
 A subset $\K\subseteq {\cR^N}$  is a   {\it cone} if $\alpha x\in \K$ whenever $\alpha>0$, $x\in \K$. Given a subset   $X\subseteq \cR^N,$  we will use  $X^{\bot}$ to denote the  {\it polar} of $X$, i.e.    $X^{\bot} \doteq \{ p\in \cR^N: \   p\cdot x \leq 0,\quad \text{for all }\; x\in X\}$.
Given an interval $I$ and $X\subseteq \cR^N,$ we write $AC(I,X)$ for the space of absolutely continuous functions, $C^0(I,X)$ for the space of continuous functions,   $L^1(I,X)$  for the Lebesgue space of $L^1$-functions, and $L^\infty(I,X)$   for the Lebesgue space of measurable, essentially bounded functions, respectively, defined on $I$ and assuming values in $X.$  As customary, we shall use 
 $\|\cdot\|_{L^\infty(I,X)}$,  and $\|\cdot\|_{L^1(I,X)}$ to denote 
 the essential supremum norm and the $L^1$-norm, respectively. When no confusion may arise, we will simply write  $\|\cdot\|_\infty$ and $\|\cdot\|_1$.   We set $\cR_+:=[0,+\infty)$ and $\cR_-:=(-\infty,0]$. Given an  integer  $ k\ge 0$ and an open subset $\Theta\subseteq\cR^N$, we  say that a function $F\colon \Theta \to \cR^N$ is {\em of class $C^k$} if it possesses continuous partial derivatives up to order $k$ in $\Theta$.   
Given a  real-valued function $F:[a,b]\to\cR$, we define  the essential infimum of $F$ as  $\ds{\text{ess}\inf}  F:=\sup\big\{r\in\cR:  \ {\rm meas}\{x\in[a,b]: \ F(x)<r\}=0\big\}$, where ${\rm meas}$ denotes the Lebesgue measure.
Finally, 
for all   $\tau_1$, $\tau_2\in (0,+\infty)$ and  for any pair $(z_1,z_2)\in C^0([0,\tau_1],\cR^N)\times C^0([0,\tau_2],\cR^N)$, let us define the distance
 \begin{equation}\label{dinfty}
\d\big((\tau_1,z_1),(\tau_2,z_2)\big) :=  
 | \tau_1 - \tau_2|+  \|\tilde z_1- \tilde z_2\|_\infty ,
  \end{equation} 
 where  for any  $z\in C^0([0,\tau],\cR^N)$,  $\tilde z$  denotes its continuous constant extension to $\cR_+$. 
 
\subsubsection{Boltyanski approximating cones}
\begin{definition}\label{approximatingcone} Let   $ Z$ be a subset of $\cR^N$ for some integer $N\ge1$.
Fix  $z\in Z$. We say that  a convex cone $\K\subseteq \cR^N$ is a {\em Boltyanski approximating  cone  for $Z$
at $z$} if there exist a convex cone  $C\subset\cR^{M} $ for some integer $
M\geq 0 $, a neighborhood $V$ of $0$ in $\cR^M,$ and a
continuous map $ F: V\cap C \to Z$ such that: $ F(0)=z $;  there exists a  linear map $L:\cR^M\to \cR^N$  verifying
$F(v) = F(0) + Lv + o(|v|)$  for all $v\in V\cap C$;  $LC=\K$. 
\end{definition}
 
\begin{definition} Let us consider two subsets $\mathcal{A}_1, \mathcal{A}_2$ of a topological space $\mathcal{X}$. If $y\in  \mathcal{A}_1\cap \mathcal{A}_2$, we say that $\mathcal{A}_1$ and $ \mathcal{A}_2$ are {\em  locally separated at $y$} provided there exists a neighborhood $V$ of $y$ such that
$
\mathcal{A}_1\cap\mathcal{A}_2\cap V = \{y\}.
$
\end{definition}

 The following open-mapping-based  result  characterizes  set-separation in terms of linear separation of  approximating cones (see e.g.  \cite{Sussmann1999geometry}). 
\begin{theorem}\label{ThmLocallySep}
Let $Z_1 $ and $Z_2 $ be subsets of $\cR^N$, $ z\in Z_1\cap Z_2 $
and let $\K_1$, $\K_2\subseteq \cR^N$ 
be   Boltyanski approximating cones  for  $Z_1 $ and $ Z_2 $, respectively, at $z$. If 
$\K_1$ or $\K_2$ is not a subspace and   $Z_1$, $Z_2$
are locally separated at $z$, then $\K_1$ and  $\K_2$ are {\em linearly separated}, namely  there exists a covector $\lambda \in\cR^N$ such that 
$
0\ne \lambda\in \K_1^\bot\cap(-\K^\bot_2).
$
\end{theorem}

  \subsubsection{Lie brackets}\label{Ss1}
 Given a fixed sequence ${\bf X}=(X_1,X_2,\ldots)$ of distinct
objects called {\em variables}, 
we call {\it words} the finite ordered strings  consisting of
the variables $X_i$, the {\it left parenthesis } $[$ and  the {\it right parenthesis} $]$, and the comma. We shall use $W({\bf X})$ to denote the set of words. For instance, $X_2X_5X_4$ and $X_3,[X_{13}[,]]X_{61}[$ are words. 

\noindent Given any
word $W\in W({\bf X})$, we use ${\rm Seq}(W)$  to denote the 
word  obtained from $W$ by deleting all left and right brackets and all
commas. We call  {\it length}  of a word $W\in W({\bf X})$, and write ${\rm Lgth}(W)$,  the
cardinality of ${\rm Seq}(W)$. For instance, if $W=[[[X_4,X_6],X_7],[X_8,X_9]]$,
${\rm Seq}(W) = X_4X_6X_7X_8X_9$ and ${\rm Lgth}(W)= 5$.

 \begin{definition} We  call  {\em formal   bracket of length $1$}  any word of length 1 and  we will say that the
{\em  bracket of two members $W_1$, $W_2$} of $W({\bf X})$ is the word $[W_1,W_2]$.

\noindent  We call {\em formal iterated brackets}  --or, simply, {\em brackets}-- of ${\bf X}$ the elements
of  the smallest subset $IB({\bf X})\subseteq W({\bf X})$ such that:  $IB({\bf X})$ contains  the  brackets of length $1$;
if $W_1$, $W_2\in IB({\bf X})$, then $[W_1,W_2]\in IB({\bf X})$; for any $b\in  IB({\bf X})$,   ${\rm Seq}(b) =X_{\mu+1},\ldots,X_{\mu+m}$ 
for some $\mu\geq 0$ and $m>0$.
\end{definition}
Notice that   ${\rm Lgth}([b_1,b_2])={\rm Lgth}(b_1)+{\rm Lgth}(b_2)$, 
for every pair of brackets $b_1$, $b_2$.
 Let $b$ be a bracket of length $m>1$. Then there
exists a unique pair $(b_1,b_2)$ of brackets
such that $b=[b_1,b_2]$. The
pair $(b_1,b_2)$ is the {\em factorization} of $b$, and 
$b_1$, $b_2$ are known, respectively, as the {\em left factor} and the
{\em right factor} of $b$. Any substring of $b$
which is itself an iterated bracket is called a {\em subbracket} of
$b$.


\begin{definition} If   $b$ is a   bracket and $S$ is a subbracket of $b$, let us define  $\dime(S;b)$
 by a backward recursion on $S$:
$\dime(b;b)\colonequals0$,\,  
$\dime(S_1;b)\colonequals\dime(S_2;b)\colonequals1+\dime([S_1,S_2];b).$
We shall refer to $\dime(S;b)$ as the {\em number of differentiations of $S$ in $b$}.
\end{definition}
It is  easy to prove that
$
\dime(S;b)
$
is equal to  the number of right brackets that occur in $b$ to
the right of $S$ minus  the number of left brackets that
occur in $b$ to the right of $S$.
 For example, if
$b=b(X_3,X_4,X_5) :=\big[X_3\,,\,[X_4,X_5]\big]$,
 then $\dime([X_4,X_5];b)=1$, $\dime(X_3;b)= 1$,  $\dime(X_4;b)=2$, $\dime(X_5;b)=2$.
 
\begin{definition}[Classes $C^{b+k}$ and $C^{b+k-1,1}$]\label{asindef}
Let $b$ be a bracket of degree $m\geq 1$, with ${\rm Seq}(b) = X_{\mu+1}\ldots
X_{\mu+m}$,  $\mu\geq 0 $. Let    ${\bf f}=(f_1,\ldots,f_{\nu})$ be a finite sequence
 of vector fields, with $\nu\geq \mu+m$,  and let $k\ge 0$ be an integer. We say that
 ${\bf f}$ {\em is of class $C^{b+k}$} if   $f_j$ is of class
$C^{\dime(X_j;b)+k}$ for each $j\in\{\mu+1,\dots,\mu+m\}$.
\medbreak\noindent
\end{definition}
For example, if
$
b=\big[[X_3,X_4],[[X_5,X_6],X_7]\big]$ and $ {\bf
f}=(f_1,\dots,f_8)$ (so $m=5$, $\nu =8$, $\mu = 2$),
then ${\bf f}\in C^{b+3}$ if, and only if, $f_3,f_4,f_7\in C^{5}$ and $f_5,f_6\in C^{6}$.
It is  easy to verify the following result.
\begin{proposition}
Let $b$, $k$, and  ${\bf f}=(f_1,\ldots,f_\nu)$ be as in Def.
\ref{asindef}, and let $(b_1,b_2)$ be the factorization of $b$.
Then ${\bf f}\in C^{b+k}$ if, and only if,
${\bf f}\in C^{b_1+k+1}$ and ${\bf
f}\in C^{b_2+k+1}$.
\end{proposition}

We are now ready to plug vector fields in place of indeterminates in a bracket.

\begin{definition} For integers $\mu \ge 0$, $m,\nu\geq 1$,  such that $\mu+m\leq \nu$, let  $b$ be a formal bracket such that ${\rm Seq}(b)= X_{\mu+1}\dots X_{\mu +m}$ and let ${\bf f}=(f_1,\ldots,f_\nu)$ be a $\nu$-tuple of continuous vector fields. Let $S$ be a subbracket of $b$.
  If ${\rm Lgth}(S)=1$, i.e.  $S=X_j$ for some $j=\mu+1,\dots, \mu+m$, we define  the vector field $S({\bf f})$
as
$S({\bf f}) \colonequals  X_j({\bf f}) \colonequals f_j$.
If  ${\rm Lgth}(S)>1$,   $S=[S_1,S_2]$, and either $S\neq b$
or,  when $S=b$, one assumes ${\bf f}\in C^b$, we set
 $
S({\bf f})\colonequals  [S_1({\bf f}),S_2({\bf f})]. $ We shall call $S({\bf f})$  the Lie bracket corresponding to the formal bracket $S$ and the sequence  ${\bf f}$ of vector fields. 
\end{definition}

We call  {\it switch-number}  of a formal  bracket  $b$ the number $r_{_b}$ defined
 recursively as:  ${r_{_b}} := 1,$ if   $b=X_j$ for some $j$;
 $ r_{_b}:= 2\big(r_{_{b_1}}+r_{_{b_2}}\big)$ if ${\rm Lgth}(b)\ge2$  and  $b=[b_1,b_2]$.  For instance, the switch-numbers of  $\big[[X_3,X_4],[[X_5,X_6],X_7]\big]$ and $[[X_5,X_6],X_7]$ are $28$ and $10$, respectively.
 We will call   {\em length} and  {\em switch-number of a Lie bracket} $B=b(f_{\mu+1},\ldots
f_{\mu+m})$ the  length and the switch-number  of the associated formal bracket $b$, respectively.


\section{The optimization  problems}\label{SecProblem}
In this section we introduce rigorously the optimization problem over $L^1$-controls  and its embedding in an impulsive problem.
\vsm
Throughout the paper we shall assume the following set of hypotheses:
\begin{itemize}
\item[{\bf (Hp)}]  {\it 
{\rm (i)} the target $\cS\subset \cR_+\times\cR^n$ is  a closed subset; the control set  $A\subset\cR^q$  is   compact;\footnote{Through minor changes one might generalize this hypothesis  with the fact that,  for every $(x,u)$, the function  $a\mapsto(f(x,a),l(x,u,a))$ is bounded.}   the {\rm unbounded control set} ${\mathcal C} \subseteq \cR^m$ is a  closed  cone of the form
$
 \mathcal{C}   = \mathcal{C} _1\times  \mathcal{C} _2,
$
where $m_1+m_2=m$,    $\mathcal{C} _1\subseteq\cR^{m_1}$ is a closed cone that  contains the lines $\{r{\bf e}_i:\ r\in\cR\}$, for $i=1,\ldots,m_1$,   and $\mathcal{C}_2\subset\cR^{m_2}$ is a  closed cone which does not contain straight lines;
 \footnote{Hypothesis (i) on $\mathcal{C}$ is by no means  restrictive, since  it can be recovered by replacing the single vector fields $g_i$ with suitable linear combinations of $\{g_1,\ldots,g_m\}$.}

{\rm (ii)}   the {\rm  drift dynamics} $ f\colon \cR^n\times  A\to\cR^{n}$ is continuous and
has continuous  partial derivatives $\displaystyle\frac{\partial f}{\partial x^1}, \dots, \frac{\partial f}{\partial x^n}$; 

{\rm (iii) } the vector fields  $g_1,\dots,g_m:  \cR^n\to\cR^{n}$  are  continuously differentiable;  
 
{\rm (iv)} the {\rm Lagrangian}  $\ell\colon \cR^n\times\C\times A \to \cR$ can be written 
  as \linebreak $\ell(x,u,a)=\ell_0(x,a)+\ell_1(x,u), $   where $\ell_0$
 and the recession function
  $$
 \hat\ell_1(x,w^0,w):= \lim_{r\to w^0 } r\ell_1\left(x,\frac{w}{r}\right),\quad\text{for all } (x,w^0,w) \in \cR^n\times \cR_+\times \C
  $$ 
  are continuous  with continuous  partial derivatives  with respect to $x.$

{\rm (v)} the {\rm  final cost} $\h:\cR\times\cR^n\to\cR$ is  continuously differentiable. 
}
 \end{itemize}
 \vsm
Clearly, by standard cut-off methods one might assume the differentiability hypotheses in (ii)-(v) only on a neighborhood of the extended optimal trajectory considered in Thms. \ref{PMPe}, \ref{PMPeho}.
 
\subsection{The original optimal control problem}
\vsm

We define the {\em set $\U$ of strict-sense controls} as
 $
\U:=
\bigcup_{T>0} \{T\}\times L^1\big([0,T] , \C\times A\big) . 
 $
 
\begin{definition}
 For any strict-sense control $(T,u,a)\in\U,$ we call $(T,u,a,x,\vz)$ a {\em strict-sense process} if  $(x,\vz)$ is  the (unique) Carath\'eodory solution to
 \bel{E2}
\left\{\begin{array}{l}
\ds \frac{dx}{dt}(t) = f(x(t), a(t)) + \ds\sum_{i=1}^m  g_i(x(t)) {u^i}(t),\\
\ds\frac{d\vz}{dt}(t) = |u(t)|,
\\ [1.5ex] 
(x,\vz)(0)=({\check x},0).\end{array}\right. \quad {\rm a.e. }\ t\in [0,T],
\eeq 
Furthermore, we say that   $(T,u,a,x,\vz)$ is {\em  feasible} if  $(T, x(T),\vz(T))\in \cS\times [0,K]$.  
 \end{definition}
 The original optimal control problem  is defined as
\begin{equation}
\label{P}\tag{P}
\left\{\begin{array}{l}
\ds\text{minimize} \ \h(T, x(T))+\int_0^T \ell(x(t),u(t),a(t))\,dt  \\
\text{over the set of feasible strict-sense processes $(T,u,a,x,\vz)$.}
\end{array}\right.
\end{equation}

\begin{definition}  We call a feasible  strict-sense process $(\bar T, \bar u, \bar a, \bar x,   \bar \va)$   a {\em local strict-sense  minimizer} of \eqref{P} if 
 there exists $\delta >0$ such that
\begin{equation}\label{min1}
\h(\bar T,\bar x(\bar T))+\int_0^{\bar T} \ell(\bar x(t),\bar u(t),\bar a(t))dt\,\leq\, \h(T, x(T)) +\int_0^T \ell(x(t),u(t),a(t))dt
\end{equation}
for every feasible strict-sense process $(T,u,a, x, \va)$ 
verifying
$
{\rm d}\Big((T,x,\va), (\bar T,\bar x,\bar \va) \Big)<\delta,
$
where ${\rm d}$ is the distance defined in \eqref{dinfty}.
 If  relation \eqref{min1} is satisfied for all feasible strict-sense processes,  
 we say that $(\bar T, \bar u, \bar a, \bar x,   \bar \va)$ is a {\em global strict-sense minimizer}.
\end{definition}

\begin{remark}\label{RemE}{\rm     
By adding   the trivial equations $\ds \frac{dx^0}{dt}(t) =1,$   $\ds \frac{d\hat x}{dt}(t) = u(t)$,  where $\hat x = (x^{n+1},\ldots, x^{n+m})$, we can allow $\ell$,  $f$, $g_i$, for $i=1,\dots,m,$ to depend on $t$ and on  the function $U(t):=\int_0^tu(\tau)\,d\tau,$ while $\h$   might depend on  $U$   as well. }
\end{remark}

\subsection{The  space-time optimal control problem}\label{Subs3.2}
We refer to the set \newline    $
 \mathcal{W}:=
\bigcup_{S>0} \{S\}\times \Big\{(w^0,w,\alpha)\in L^\infty([0,S], \cR_+\times \C\times A): \,  {\text{ess}\inf}  (w^0+|w|)>0\Big\} 
 $ as  the {\em set of space-time controls}.

\begin{definition}
 For any   $(S, w^0,w,\alpha)\in \mathcal{W},$  we  say that $(S,w^0,w,\alpha,y^0,y,\beta)$ is  a {\em space-time process} if  $(y^0,y,\beta)$ is  the unique Carath\'eodory solution of
 \begin{equation}
 \label{extended}
 \left\{
\begin{split}
\frac{d{y^0}}{ds} (s) &= w^0(s),  \\
\frac{dy}{ds} (s) & = f( y(s),\alpha(s))w^0(s)+ \sum_{i=1}^{m}g_{i}( y(s))w^i(s),\\
\frac{d\beta}{ds} (s) & = |w(s)|, \\
 (y^0,y , &\beta)(0)=(0,\xbo,0).
\end{split}
\right. \quad {\rm a.e.}\, s\in [0,S], \end{equation}
We say that $(S,w^0,w,\alpha,y^0,y,\beta)$ is {\em  feasible} if $(y^0(S),y(S),\beta(S))$ belongs to $\cS\times [0,K]$. 
 \end{definition}
We define the {\it extended or space-time problem} as
\begin{equation}
\label{Pe}\tag{P{\tiny s-t}}
\left\{\begin{array}{l}
\ds\text{minimize} \ \h(y^0(S), y(S) )  +\int_0^S \ell^e\big((y,w^0, w,\alpha)(s)\big)ds  \\
\text{over  feasible space-time processes 
$(S,w^0,w,\alpha,y^0,y,\beta)$,}
\end{array}\right.
\end{equation}
where the {\it extended Lagrangian} $ \ell^e$ is given by
$$
 \ell^e(x,w^0,w,a) := \ell_0(x,a)w^0 + \hat\ell_1(x,w^0,w), \quad \text{for } (x,w^0,w,a)\in\cR^n\times  \cR_+\times \C\times A.
$$
 \begin{definition}\label{emin} A feasible space-time process   $(\bar S, \bar w^0,\bar w,\bar\alpha, \bar y^0,\bar  y,\bar \beta)$   is said to be a  {\em local minimizer for the space-time problem \eqref{Pe}} if 
  there exists $\delta >0$ such that
\begin{equation}
\label{min}
\ds\h((\bar y^0,\bar y)(\bar S)) +\int_0^{\bar S} \ell^e((\bar y,\bar w^0,\bar  w,\bar \alpha)(s))\,ds \leq  
\ds  \h((y^0,y)(S)) +\int_0^S \ell^e((y,w^0, w,\alpha)(s))\,ds
\end{equation}
 for all feasible $(S,w^0,w,\alpha,y^0,y,\beta)$ satisfying
$
{\rm d}\Big((S,y^0, y ,\beta) , (\bar S,\bar y^0, \bar y ,\bar \beta)\Big)<\delta,
$
where ${\rm d}$ is as  in \eqref{dinfty}.  If \eqref{min} is satisfied for all feasible space-time processes, 
we call
 $(\bar S, \bar w^0,\bar w,\bar\alpha, \bar y^0,\bar  y,\beta)$  a {\em global space-time  minimizer}.  
\end{definition}
Observe that the space-time  system \eqref{extended} is  {\it rate-independent}. Precisely, given a strictly increasing, surjective,  and  bi-Lipschitzian function  $\sigma:[0, S]\to [0, \tilde S]$, \linebreak $(\tilde S,\tilde w^0,\tilde w,\tilde \alpha,\tilde y^0,\tilde y,\tilde \beta)$ is a space-time process  if,  and only if,    $(S,w^0,w,\alpha,y^0,y,\beta)$  given by
 $
\ds(w^0,w):=\Big((\tilde w^0,\tilde w)\circ\sigma\Big)\,\frac{d\sigma}{ds}$, $(\alpha,y^0,y,\beta):=\left(\tilde \alpha,\tilde y^0,\tilde y,\tilde \beta\right)\circ\sigma$ \footnote{Since  every $L^1$-equivalence class contains  Borel measurable representatives, we  tacitly assume that all $L^1$-maps we are considering are Borel measurable    when necessary.
}
is a space-time process of  \eqref{extended} (see \cite[Sect. 3]{MottaRampazzo1995}). In this case, $(\tilde S,\tilde w^0,\tilde w,\tilde \alpha,\tilde y^0,\tilde y,\tilde \beta)$ is feasible if, and only if, $(S, w^0,w,\alpha, y^0,y,\beta)$ is feasible, and the associated costs coincide.
Let us call {\it equivalent}  any two space-time processes $(\tilde S,\tilde w^0,\tilde w,\tilde \alpha,\tilde y^0,\tilde y,\tilde \beta)$,  $(S,w^0,w,\alpha,y^0,y,\beta)$ as above. The following result is quite  straightforward:
\begin{lemma}\label{equiv}
A feasible space-time process   $(\bar S, \bar w^0,\bar w,\bar\alpha, \bar y^0,\bar  y,\bar \beta)$   is  a local 
  (resp., a global) minimizer for the space-time problem \eqref{Pe} if, and only if, every equivalent space-time process is  a local  (resp., a global) minimizer, and   the costs coincide.
\end{lemma}  
As a consequence, the extended problem can be regarded  as a  problem on the quotient space. Therefore,  without loss of generality, one can replace a  minimizer  with its {\it canonical parameterization,} defined as follows: 
 
\begin{definition} We say that $(S_c,w_c^0,w_c,\alpha_c, y_c^0,y_c,\beta_c )$  is the  {\rm canonical parameterization} of a  space-time process  $(S,w^0,w,\alpha, y^0,y,\beta)$ if $$(w^{0}_c,w_c):=\Big((w^0,w)\circ \sigma^{-1}\Big)\,\frac{d\sigma^{-1}}{ds},  \quad (\alpha_c, y_c^0,y_c,\beta_c):=(\alpha, y^0,y,\beta)\circ \sigma^{-1},$$ where
$
   \sigma(s):=\int_0^s\left(w^0 (r)+\left|w(r)\right|\right)dr$, $s\in[0,S]$, $S_c:= \sigma(S)=y^0(S)+\beta(S).
$
\end{definition}
Note that  $w^{0}_c   (s)+\left|w_c (s)\right|=1$ for a.e. $s\in[0,S_c]$. We introduce the subset of  {\it canonical  space-time controls}
$$
  \mathcal{W}_c:=\left\{(S,w^0,w,\alpha)\in  \mathcal{W}: \,  w^0(s)+|w(s)|=1\quad \text{a.e. $s\in[0,S]$}\right\}, 
$$
and call   {\it canonical} also   the corresponding space-time processes. One can  easily verify that  a canonical space-time process coincides with its canonical parameterization.

\subsection{The space-time embedding}
The original control system \eqref{E2} can be embedded  into the space-time system \eqref{extended}. Precisely,  by the chain rule, given a strict-sense process $(T,u,a,x,\va)$, by setting 
\bel{seq0}
\ds\sigma(t):=\int_0^t (1+|u(\tau)|)\,d\tau,\quad  S:=\sigma(T),\quad y^0:=\sigma^{-1}:[0,S]\to[0,T],
\ee
 one obtains that 
\bel{seq}
(S,w^0,w,\alpha,y^0,y,\beta):=\left(S,  \frac{dy^{0}}{ds}, (u\circ y^{0})\cdot \frac{dy^{0}}{ds},  a\circ y^{0}, y^{0},x\circ y^{0}, \va\circ y^{0}\right) 
\eeq
  is a (canonical) space-time process   with $w^0>0$ a.e..
 Conversely, given a space-time process  $(S,w^0,w,\alpha,y^0,y,\beta)$ with $w^0>0$  a.e.,
   the increasing surjective  function $y^0:[0,S]\to[0,T]$, has an absolutely continuous
inverse $\sigma:[0,T]\to [0,S]$ (see e.g. \cite{Fed69}),  and 
 $(T,u,a,x,\vz):=\left(T, (w\circ\sigma)\frac{d\sigma}{dt}, \alpha\circ\sigma, y\circ\sigma,\beta\circ\sigma\right)$  is  a  strict-sense process. Hence, the family of strict-sense processes  can be identified with the subfamily of space-time  processes $(S,  w^0,w,\alpha, y^0,y,\beta)$ having $w^0>0$ a.e..
  
The impulsive, space-time extension of the original optimal control problem consists in  allowing  the  control $w^0$  to vanish in a set of positive measure. The $s$-intervals where $w^0$ vanishes represent the `impulses', namely the $s$-intervals of instantaneous evolution of both the control  and the state (see  e.g. \cite{BressanRampazzo1988, MottaRampazzo1995}).\footnote{Let us point out that one can equivalently give a $t$-based description of this extension using bounded variation trajectories as in   
\cite{AronnaRampazzo2015a,MottaSartori2018,AruKarPer10}.} 

 
 The notions of strict-sense and space-time  local minimizer are consistent, as stated in the following  easy consequence of Lemma \ref{equiv} above and  \cite[Prop. 2.7]{AronnaMottaRampazzo2019}:  
\begin{lemma}
\label{Min=} 
A  process  $(\bar T, \ub,\ab,\xb,\bar\va)$   is a  {\em strict-sense local minimizer}  for problem \eqref{P} if, and only if, $(\bar S,\bar w^0,\bar w,\bar \alpha,\yb^0,\yb,\bar\beta)$ defined as in \eqref{seq0}-\eqref{seq}  is a space-time local minimizer for \eqref{Pe}  among  feasible space-time processes with $w^0>0$ a.e..
\end{lemma}

\section{A First Order Maximum Principle}\label{SNC}
Due to the rate-independence of the space-time control system discussed in Subsection \ref{Subs3.2}, we can always assume that a local minimizer $(\bar S, \bar w^0,\bar w,\bar\alpha, \bar y^0,\bar  y,\bar\beta)$  for   \eqref{Pe} is canonical. 
 
Let us  set 
\bel{DS+}
W:=\{(w^0,w)\in\cR_+ \times \C: \,  w^0+|w|=1\} .
\eeq
Let us   consider the  {\it unmaximized Hamiltonian} $H: \cR^{n+1+n+1+1}\times\cR_+\times\C\times A \to \cR$  and the  {\em Hamiltonian} ${\bf H}: \cR^{n+1+x+1+1} \to \cR$, defined by setting  
$$
H(x,p_0,p,\pi,\lambda,w^0,w,a):= p_0w^0 + p\cdot\Big(f(x,a) w^0 +  \sum_{i=1}^{m}  g_{i}(x) w^i\Big) + \pi | w| - \lambda\ell^e(x,w^0,w,a),
$$
$${\bf H}(x,p_0,p,\pi,\lambda
) := \ds\max_{(w^0,w,a)\in W\times A}  H(x,p_0,p,\pi,\lambda,w^0,w,a).
$$ 
{\begin{theorem}[\bf First Order Maximum Principle]
\label{PMPe} Let $(\bar S,\bar w^0,\bar w,\bar\alpha,  \bar y^0,\bar  y,\bar\beta)$  be a  canonical  local minimizer
for  the space-time problem \eqref{Pe}. Then, for every \linebreak Boltyanski  approximating cone   $\Gamma$    of the target $\cS$ at $(\bar y^0,\bar y)(\bar S)$, there exists  a multiplier   $(p_0, p , \pi,\lambda)\in \cR\times AC\left([0,\bar S],\cR^{ n}\right)\times \cR_-
\times \cR_+$ verifying: 
\begin{itemize}
\item[(i)]
 {\sc (non-triviality)} 
\begin{equation} 
\label{fe1}
(p_0, p , \lambda) \not= (0, 0,0) \,.
\end{equation}
Furthermore, if 
   $ \bar y^0(\bar S)>0$, then \eqref{fe1} can be strengthened  to 
\begin{equation}\label{strongfe1}
  (p  , \lambda) \not= (0,0). 
  \end{equation}
  \item[(ii)] {\sc (non-tranversality)}  
 \begin{equation}
 \label{fe4}
 (p_0,p(\bar S),\pi) \in \left[-\lambda\left(\frac{\partial\Psi}{\partial t}\big( (\bar y^0, \bar y)(\bar S)\big)\,,\,\frac{\partial\Psi}{\partial x}\big( (\bar y^0, \bar y)(\bar S)\big) \right)- \Gamma^\perp\right]\times J,
 \end{equation}
  where $J:=\{0\}$ if $\bar\beta(\bar S) < K $, and $J:=(0,+\infty)$ 
 if $\bar\beta(\bar S) = K$.\footnote{It is tacitly meant that, as an approximating cone to the $(T,x,\vz)$-target $\cS\times [0,K]$ at $(\bar y^0,\bar y,\bar\beta)(\bar S)$, one chooses  $\Gamma\times \cR$   if $\bar\beta(\bar S)<K$ and $\Gamma\times (-\infty,0]$   if $\bar\beta(\bar S)=K$. In particular, 
$(\Gamma\times \cR)^\perp = \Gamma^\perp\times\{0\}$  if $\bar\beta(\bar S)<K$ and $(\Gamma\times (-\infty,0])^\perp = \Gamma^\perp\times\cR_+$ when $\bar\beta(\bar S)=K$.}  In particular,   
 \bel{piestzero}\pi = 0 \quad\hbox{provided}\quad \bar\beta(\bar S)< K. \eeq 
\item[(iii)] {\sc (adjoint equation)}  The path $ p $ solves, for a.e. $s\in [0,\bar S]$, 
\begin{equation}
\label{fe2}
\displaystyle  \frac{dp}{ds} (s)\,=\,-\frac{\partial H}{\partial x}\left(\bar y(s),p(s),\pi,\lambda,\bar w^0(s),\bar w(s),\bar \alpha(s)\right).
\end{equation} 
\item[(iv)] {\sc (First order maximization)} For a.e. $s\in [0,\bar S]$, 
\begin{equation}\label{fe3}
\begin{array}{l}
H\Big(\bar y(s), p_0 , p(s),\pi, \lambda,\bar w^0(s),\bar w(s),\bar\alpha(s)\Big)= 
 {\bf H}\Big(\bar y(s), p_0 , p(s),\pi, \lambda\Big).
\end{array}
\end{equation}
\item[(v)] {\sc (Vanishing of the Hamiltonian)} 
\bel{engine}
 {\bf H}\Big(\bar y(s), p_0 , p(s),\pi, \lambda\Big)=0, \quad \text{for all }s\in[0,\bar S].
\eeq
\end{itemize}
\end{theorem}
 
\begin{proof} The  Pontryagin Maximum Principle based on Boltyanski approximating cones (see 
 e.g.  \cite{Sussmann1999geometry, SchaettlerLedzewicz2012})
 yields  the existence of a  multiplier $(p_0, p , \pi,\lambda)\in \cR\times AC\left([0,\bar S],\cR^{ n}\right)\times  \cR_-
\times \cR_+$ verifying the non-transversality condition \eqref{fe4}, the adjoint equation \eqref{fe2}, the  maximum relation \eqref{fe3}, the conservation \eqref{engine}, and the non-triviality condition $
(p_0, p , \pi,\lambda)\neq 0.
$
So, it remains to prove  the strengthened non-triviality condition \eqref{fe1}. This can be done by using  the same elementary  arguments as in the proof of \cite[Theorem 3.1]{MottaRampazzoVinter2019}.\footnote{The fact that 
in  \cite{MottaRampazzoVinter2019} 
one  makes  use of the limiting normal cone instead of the polar of the Boltyanski  cone  plays no role in the proof of this result.}
\end{proof}

\begin{definition}\label{abnormal}
A  process  $(\bar S, \bar w^0,\bar w,\bar\alpha, \bar y^0,\bar  y,\bar\beta)$ is called  an {\rm extremal} if it obeys the conditions in Theorem \ref{PMPe} 
  for some   multiplier  $(p_0, p, \pi,\lambda)$. If there is a choice of the multiplier  with $\lambda=0,$ then the extremal $(\bar S, \bar y^0,\bar  y,\bar\beta,\bar w^0,\bar w,\bar\alpha)$ is called {\rm abnormal}, otherwise it is called  {\rm normal}. Finally, the extremal is said to be {\rm strictly abnormal} if every choice of the multiplier $(p_0, p , \pi,\lambda)$ verifies $\lambda=0$. 
\end{definition}
 
When $\ell_1(x,\cdot)$ is positively  $1$-homogeneous, so that  for any $(x,w^0,w,a)\in\cR^n\times\cR_+\times\C\times A$ one has
 $\ell^e(x,w^0,w,a) = \ell_0(x,a)w^0 + \ell_1(x,w)$,  
let us  define the  {\it drift Hamiltonian} ${\bf H}^{\rm (dr)}$ and the  {\it impulse Hamiltonian} ${\bf H}^{\rm (imp)}$:
$$ 
\begin{array}{l}
{\bf H}^{\rm (dr)}\Big( x,p_0,p,\lambda\Big):=\ds\max_{a\in A}\Big\{p_0+p \cdot f(x,a) -\lambda \ell_0(x,a)\Big\}, \\
{\bf H}^{\rm (imp)}\Big(x,p,\pi,\lambda\Big):=\ds\max_{w\in \C,|w| =1}\left\{p\cdot\sum_{i=1}^{m}  g_{i}(x) w^i + \pi    -\lambda \ell_1 (x,w)\right\}.
\end{array}
 $$  
 
\begin{corollary}
\label{cor2first}  Let  $\ell_1(x,\cdot)$ be positively  $1$-homogeneous and let\\ $(\bar S, \bar w^0,\bar w,\bar\alpha, \bar y^0,\bar  y,\bar\beta)$   be a canonical  extremal obeying the conditions in Theorem \ref{PMPe} 
  for some  multiplier  $(p_0, p, \pi,\lambda)$. Then  there exists a zero-measure  subset ${\mathcal N}\subset [0,\bar S]$ such that,  for  every   $s\in[0,\bar S]\setminus {\mathcal N},$ one has 
   \begin{align}
   &H\Big( \bar y(s), p_0 , p(s),\pi, \lambda,\bar w^0(s),\bar w(s),\bar\alpha(s)\Big)  =  {\bf H}\Big(\bar y(s), p_0 , p(s),\pi, \lambda\Big)   \label{dueH}  \\
 &\quad\qquad\qquad=\max\left\{{\bf H}^{\rm (dr)}(\bar y(s),p_0,p(s),\lambda)\,,\, {\bf H}^{\rm (imp)}(\bar y(s),   p(s),\pi,\lambda)\right\}=0,\notag
\\
& \bar w^0(s)\Big[p_0 + p(s)\cdot f(\bar y(s),\bar\alpha(s)) -\lambda \ell_0(\bar y(s),\bar\alpha(s))\Big] =0,  \label{cdr}
\\
&  p(s)\cdot\sum_{i=1}^{m}  g_{i}(\bar y(s)) \bar w^i(s) + \pi |\bar w(s)| - \lambda \ell_1 (\bar y(s),\bar w(s))=0.  \label{cimp}
\end{align}
 In particular,  if for some    $s\in[0,\bar S]\setminus {\mathcal N}$ one has 
\vsm
\noindent i) \, ${\bf H}^{\rm (dr)}(\bar y(s),p_0,p(s),\lambda)<0$, then $\bar w^0(s)=0$  and   
$$
 p(s)\cdot   \sum_{i=1}^{m}  g_{i}(\bar y(s)) \bar w^i(s) + \pi - \lambda\ell_1(\bar y(s),\bar w(s))={\bf H}^{\rm (imp)}(\bar y(s), p_0 , p(s),\pi,\lambda)=0;
$$
ii) \, ${\bf H}^{\rm (imp)}(\bar y(s),  p(s),\pi,\lambda)<0$, then $\bar w(s)=0$  and
$$
 p_0 + p(s)\cdot f(\bar y(s),\bar\alpha(s)) - \lambda\ell_0(\bar y(s),\bar\alpha(s))={\bf H}^{\rm (dr)}(\bar y(s),p_0,p(s),\lambda)=0.
$$
\end{corollary}

 \begin{proof}
 By \eqref{engine} it follows that for every $s\in[0,\bar S]$, one has
 $$
 p_0w^0 + p(s)\cdot\Big(f(\bar y(s),a) w^0 +  \sum_{i=1}^{m}  g_{i}(\bar y(s)) w^i\Big) + \pi |w| - \lambda\Big(\ell_0(\bar y(s),a)w^0+\ell_1 (\bar y(s),w)\Big)\le 0
 $$
 for all  $(w^0,w,a)\in W\times A$.
Now, by choosing $w=0$ one gets that $w^0=1$ and
$$
 p_0 + p(s)\cdot f(\bar y(s),a)  - \lambda\ell_0(\bar y(s),a)\le 0,  \quad \text{for all } a\in A,
 $$ 
 while taking $w^0=0$ one obtains  
  $$
 p(s)\cdot \sum_{i=1}^{m}  g_{i}(\bar y(s)) w^i + \pi  - \lambda\ell_1 (\bar y(s),w)\le 0,  \quad \text{for all } (w,a)\in \C\times A, \ |w|=1.
 $$
 Therefore, ${\bf H}^{\rm (dr)}(\bar y(s),  p(s),\pi,\lambda)\le0$ and ${\bf H}^{\rm (imp)}(\bar y(s),  p(s),\pi,\lambda)\le0$.  In fact,  it must be that
  $\max\left\{{\bf H}^{\rm (dr)}(\bar y(s),p_0,p(s),\lambda)\,,\, {\bf H}^{\rm (imp)}(\bar y(s),   p(s),\pi,\lambda)\right\}=0$, since, otherwise, both Hamiltonians would be negative, which   contradicts \eqref{engine}.  By taking  ${\mathcal N}\subset [0,\bar S]$ to  be the zero-measure subset such that the first order maximization \eqref{fe3} is verified in $[0,\bar S]\setminus {\mathcal N}$, we get   \eqref{dueH}. 
 If  $s\in [0,\bar S]\setminus {\mathcal N}$, by  \eqref{fe3}, \eqref{engine} one has that
$$\begin{array}{l}
\ds \bar w^0(s)\Big[p_0 + p(s)\cdot f(\bar y(s),\bar\alpha(s)) -\lambda \ell_0(\bar y(s),\bar\alpha(s))\Big] \\ 
 \ds \qquad\qquad+\Big[ p(s)\cdot\sum_{i=1}^{m}  g_{i}(\bar y(s)) \bar w^i(s)) + \pi |\bar w(s)| - \lambda \ell_1 (\bar y(s),\bar w(s)) \Big]= 0.
 \end{array}
$$
Since the above argument implies that both terms in this equality are nonpositive, they necessarily vanish,  namely \eqref{cdr}  and   \eqref{cimp} are verified.

 To prove i), suppose  ${\bf H}^{\rm (dr)}(\bar y(s),  p(s),\pi,\lambda)<0$. Then \eqref{cdr}  implies  $ \bar w^0(s)=0$, so that $|\bar w(s)|=1$  and the thesis i) follows by \eqref{cimp}.  
 Finally, to prove ii) assume  that  ${\bf H}^{\rm (imp)}(\bar y(s),  p(s),\pi,\lambda)<0,$ then $\bar w(s)=0$ due to \eqref{cimp} and in view of the positive 1-homogeneity of  $H$ w.r.t. $(w^0,w)$. Hence $ \bar w^0(s)=1$ and  \eqref{cdr} yields ii). 
 \end{proof}

\begin{remark}
Under the same hypotheses of Cor. \ref{cor2first}, ${\bf H}^{\rm (dr)}(\bar y(s),p_0,p(s),\lambda)=0 $ for all $s\in[s_1,s_2]$ as soon as $s_1,s_2\in[0,\bar S]$ are such that   $\bar w^0(s) >0$ for a.e.   $s\in [s_1,s_2]\subseteq[0,\bar S]$.
\end{remark}}

\begin{corollary}
\label{cor1first}  
Let $(\bar S,\bar w^0,\bar w,\bar\alpha, \bar y^0,\bar  y,\bar\beta)$  be a  canonical 
extremal for the space-time  problem \eqref{Pe} and let  $(p_0, p, \pi,\lambda)$ be
 a corresponding multiplier. If 
\bel{qab}
 \pi=0 \ \ \text{and} \ \ \lambda \ell^e(\bar y(s),0,\pm{\bf e}_i,a)=0, \quad   \text{for all } s\in[0,\bar S],\, i=1,\ldots,m_1, \footnote{By the definition of $\ell^e$, it is clear that the quantities   $\ell^e(\bar y(s),0,\pm{\bf e}_i,a)$ do  not depend on $a$.}
\eeq
  then $ p(s)\cdot g_i(\bar y(s))=0$ for all $s\in[0,\bar S]$, $i=1, \ldots,m_1$.
\end{corollary}

\begin{proof}  By \eqref{engine} it follows that  for every $s\in[0,\bar S]$ and  $(w^0,w,a)\in W\times A$,   one has 
$
 p_0w^0 +  p(s)\cdot\Big(f(\bar y(s),a) w^0 +  \sum_{i=1}^{m}  g_{i}(\bar y(s)) w^i\Big)- \lambda\ell^e(\bar y(s),w^0,w,a) \le 0.  
$
 Therefore, choosing $w^0=0$ and $w=\pm{\bf e}_i$ for any $i=1,\dots,m_1$, one gets  the thesis.
\end{proof}

\begin{remark}{\rm From Theorem \ref{PMPe},   one has $\pi=0$ as soon as $\bar\beta(\bar S)<K$. Moreover, the hypothesis  $\lambda \ell^e(\bar y(s),0,\pm{\bf e}_i,a)=0$  is obviously satisfied when the extremal $(\bar S, \bar w^0,\bar w,\bar\alpha, \bar y^0,\bar  y,\bar\beta)$ is abnormal and one chooses  $\lambda=0$,     or if $\hat\ell_1(x,0,w)=0$ for all 
$(x,w)\in\cR^n\times (\cR^{m_1}\times\{0\}{^{m_2}})$.
 (This includes, in particular, the case $\ell_1\equiv 0$, as in the minimum time problem, where $\ell_0\equiv 1$.)}
\end{remark}

 \section{A Higher Order Maximum Principle}\label{SMPho}
 Let us begin with a regularity notion for Lie brackets of the vector fields $g_1,\ldots,g_{m_1}$.
 \begin{definition}\label{DB}
For every   integer $k\ge0$, we say that a vector field $B$ is a {\rm $C^k$-admissible  Lie bracket} if $B=b(F_1,\ldots,F_q)$, where $b$ is a formal bracket  and $(F_1,\ldots,F_q)$ is a $q$-tuple of  class $C^{b+k}$ of vector fields in $\{g_1,\ldots,g_{m_1}\}$  (see Def.\ref{asindef}).  We will use $\B^k$ to denote the set of $C^k$-admissible Lie brackets of length   $\ge2$.
\end{definition}

\subsection{Higher order conditions}
\begin{theorem}[\bf Higher Order Maximum Principle]
\label{PMPeho}
Assume that  hypothesis {\bf (Hp)} is satisfied  with $\hat\ell_1(\cdot,0,\cdot)\equiv 0$. Let
  $(\bar S, \bar w^0,\bar w,\bar\alpha,  \bar y^0,\bar  y,\bar\beta)$  be a canonical  local minimizer
for the space-time problem  \eqref{Pe} that verifies $\bar\beta(\bar S)<K$. Then, for every  Boltyanski approximating cone  $\Gamma$    of the target $\cS$ at $(\bar y^0,\bar y)(\bar S)$, there exists  a multiplier  
 $(p_0, p,\pi, \lambda)\in \cR\times AC\left([0,\bar S],\cR^{ n}\right)\times \cR_-
\times \cR_+$  with $\pi=0$     that satisfies  all the conditions of Theorem \ref{PMPe} and, moreover, verifies
\begin{gather}
 p(s)\cdot g_i(\bar y(s))=0,   \qquad \text{for all } s\in[0,\bar S],\,  i=1,\dots,m_1,  \label{pg0hi}
 \\
  p(s)\cdot B(\bar y(s))=0,   \qquad \text{for all } s\in [0,\bar S],\, B\in\B^0. \label{pg000hi}
\end{gather}
\end{theorem}

The proof of this theorem is postponed to Section \ref{SecProof}.  

  \begin{remark} Requiring the condition   $\hat\ell_1(\cdot,0,\cdot)\equiv 0$  is crucial for the general validity of Theorem \ref{PMPeho}. Otherwise, the variations  corresponding to brackets of length $h\geq 2
$ would produce a  perturbation of order $\eps^{\frac 1h}$ of the cost variable, so having infinite derivative w.r.t. $\eps$.
Since the same variation would produce a change of order $\eps$ in the dynamical variables,  the separation  Theorem \ref{ThmLocallySep} turns  out to be  not applicable.  However, as soon as the minimizer is strictly  abnormal,  one might be able to deduce some results involving Lie brackets also  for   the case $\hat\ell_1(\cdot,0,\cdot)\ne 0$ as well, possibly via some higher-order open mapping argument. This would be  similar to what happens in the case of sub-Riemannian geometry  \cite{Agrachev2019comprehensive}. We leave this issue as an open question.
\end{remark}

   \begin{remark}\label{rcs} Since we obtained the higher order necessary conditions under the only prerequisite that the involved Lie brackets are continuous, one might wonder to which extent such a regularity hypothesis can be further weakened.  
For instance, one might prove an extension of Theorem \ref{SMPho} by means of  {\em set-valued Lie brackets} of non smooth vector fields, as studied in \cite{RampazzoSussmann2001,RampazzoSussmann2007,FeleqiRampazzo2017}.
 \end{remark}

In the sequel we will use the notation $f_a(\cdot):=f(\cdot,a)$.

\begin{corollary}\label{Coreho}
Assume that hypothesis {\bf (Hp)} is satisfied  with $\hat\ell_1(\cdot,0,\cdot)\equiv 0$,  and 
let $(\bar S, \bar w^0,\bar w,\bar\alpha,  \bar y^0,\bar  y,\bar\beta)$  be a canonical  local minimizer
of   \eqref{Pe} that verifies $\bar\beta(\bar S)<K$. Given a  Boltyanski approximating cone  $\Gamma$   of the target $\cS$ at $(\bar y^0,\bar y)(\bar S)$, let   $(p_0, p ,\lambda)$  be a multiplier as in Theorem 
 \ref{PMPeho}. Then,  for any  Lie  bracket $B\in\B^1\cup\{g_1,\dots,g_m\}$,
  one has \footnote{{\em I.e.}, $B$ is a $C^1$-admissible Lie bracket (possibly of length $1$), see Definition \ref{DB}.} 
 \begin{multline}\label{MP111new}
   p(s)\,\cdot\, \Big(\big[f_{\bar\alpha(s)}, B\big](\bar y(s))\bar w^0(s) + \ds\sum_{j=m_1+1}^{m}
  \big[g_j,B\big](\bar y(s))\bar w^j(s)\Big) \\ =\ds-\lambda
\frac{\partial \ell^e}{\partial x}(\bar y(s),   \bar w^0(s), \bar w(s),\bar\alpha(s))\cdot B(\bar y(s)),
\end{multline}
 for  a.e.   $s\in [0,S]$. In particular, if $m_1=m$ and  the condition
\bel{flat0}
\  \lambda\, \frac{\partial \ell^e}{\partial x}(\bar y(s),\bar w^0(s), \bar w(s),\bar\alpha(s))\cdot B(\bar y(s))=0  \quad\text{for a.e. $s\in[0,\bar S]$ }
\eeq
  is satisfied, one obtains
 \beq\label{MP111newpart}   p(s) \cdot  \Big(\big[f_{\bar\alpha(s)}, B\big](\bar y(s)) \Big)\bar w^0(s)=0 \quad\text{for a.e. $s\in[0,\bar S]$.}
\eeq
  \end{corollary}

\begin{proof}
Condition \eqref{MP111new} can be obtained by differentiating  \eqref{pg0hi} or \eqref{pg000hi} and remembering  that the derivative of $p$ verifies the adjoint equation \eqref{fe2}.
  \end{proof}
  
  \begin{remark}{\rm  Condition \eqref{flat0} is satisfied for all $s\in[0,\bar S]$
   in at least two important situations, namely in the {\it abnormal case}, i.e. if $\lambda=0$, or when
$\ell = \ell_0 + \ell_1(u),$
with  $\ell_0$, $\ell_1$ independent of $x$ and  $\hat\ell_1(0,w)\equiv 0$ (for instance, in the minimum time problem).
   }\end{remark}


\begin{remark}[\sc Linear Systems]\label{remlin}{\rm  Let us consider the linear system
$$
\frac{dx}{dt}=  
C x+E u, \qquad u\in\cR^m,
$$
where $C$, $E$ are  $n\times n$ and  $n\times m$  real matrices,  respectively.    For the vector fields    
$f(x)=:Cx$,  and   $g_i$, where $g_{ij}:=E_{ji}$  for each   $i=1,\dots,m$, $j=1,\dots,n$, 
  the conditions involving Lie brackets of the $g_i$  become trivial, since $[g_i,g_j] = 
 0$. However, because of the linearity of $f(x) = Cx$, further higher order conditions can be trivially deduced under assumption 
\eqref{flat0}.  Indeed,  condition \eqref{pg0hi}
  reduces to 
\bel{c2}
p(s)\cdot E =0,  \quad\text{for all } s\in[0,\bar S],
\eeq
while, due to \eqref{flat0}, the adjoint equation  now reads
$
 \frac{dp}{dt}= -p\cdot C 
$.
Therefore by   differentiating \eqref{c2} $n-1$ times,  we get the {\it additional necessary conditions}
  $p\cdot [f,g_i]\,=\, p\cdot  [f,[f,g_i]]\,=\, p\cdot [f,[f,[\dots,[f,g_i]\dots]]]=0$,   for all  $i=1,\dots,m$,
which correspond to the $n-1$ matrix relations
\bel{hobell}
  p\cdot CE=0,\ \ p\cdot  C^2E=0,\  \dots,  \  p\cdot  C^{n-1}E=0.
\eeq}
\end{remark}

\begin{remark}
As  observed in the Introduction, some motivations  for studying impulsive systems are to be found  in Classical Mechanics. This  is  a reason why one  might be interested in extending previous results to  manifolds. Actually, such an  extension does not present any special difficulty, in that the thesis of Theorem \ref{PMPeho} has a {\em chart-independent character.}  
\end{remark}

\subsection{Fully impulsive processes}   
The necessary conditions established  in Theorems  \ref{PMPe} and \ref{PMPeho} can be used to get information on the structure of optimal trajectories: for instance,  one can wonder under which conditions  an optimal trajectory is a finite concatenation of   impulsive and non impulsive paths (as it occurs  e.g.  in the example in \cite{AronnaMottaRampazzo2019}).  Though an accurate investigation in this direction goes beyond the objectives of this paper, let us  highlight some   rank  conditions that happen to 
   force an optimal process $(\bar S, \bar w^0,\bar w,\bar\alpha,  \bar y^0,\bar  y,\bar\beta)$  to be {\it  fully impulsive}. By this we mean that  it evolves in zero time,  namely $\bar y^0(\bar S)=0$, or, equivalently,  $\bar w^0=0$ a.e. on $[0,\bar S].$
   
 \vsm
To state our result,   we introduce two rank-type assumptions:  

\vsm
{\it \noindent {\bf(I)}  
 {\sc $C^0$-Pointwise Rank Conditions} at $x\in\cR^n$. 
 
{\rm({\bf I}.1)$_x$} there exists an integer $r\geq 0$ and iterated   Lie  brackets $B_1,\ldots, B_{r}\in\B^0$  such that 
\bel{r1}
 {\rm span}\left\{B_1, \ldots, B_{r},g_1,\dots,g_{m_1}\right\}(x) = \cR^n ;  \  \footnote{We mean that $\{\zeta_1,\ldots,\zeta_N\} =\emptyset$ as soon as $N=0$.}
\eeq

{\rm ({\bf I}.2)$_x$} for every $a\in A$, there exist   integers $r\geq 0$, $k\geq 0$, and iterated Lie  brackets $B_1,\ldots, B_r\in\B^0$,
$\hat B_1,\ldots, \hat B_{k}\in\B^1$,  such that 
 \bel{fc0rank}
  {\rm span}\big\{B_1,..., B_r, [f_a,\hat B_1],..., [f_a,\hat B_k],  g_1,...,g_{m_1},  [f_a,g_1],...,[f_a,g_{m_1}]\big\} (x)= \cR^n.  
\eeq

\noindent {\bf(II)} {\sc Kalman Controllability Condition}.
The system is linear and  the  Kalman Controllability Condition    is verified, namely 
$$\frac{dx}{dt}=  C x+E u,  \ \text{ and } \ {\rm rank}  (E\,\,\,\,CE \,\,\,\, C^2E\,\,\,\, \ldots\,\,\,\, C^{n-1}E) = n ,
$$
where $C$, $E$ are  $n\times n$ and  $n\times m$  real matrices,  respectively.  
}
\vsm 
We will consider the following  assumption:
 \vsm
 \begin{itemize}
\item[{\bf (Hp1)}]  {\it Hypothesis {\bf (Hp)} holds and, moreover,
 {\rm (i)}  the target is time-invariant, namely $\cS = \cR\times \hat\cS$, with $\hat\cS\subseteq\cR^n$;
{\rm (ii)}  the final  cost $\Psi$ is time-independent;
{\rm (iii)}  the Lagrangian $\ell$ is strictly positive and $\hat\ell(\cdot,0,\cdot)\equiv 0$.
}
\end{itemize}
 \begin{theorem}\label{ThmFully} 
 Let us assume  hypothesis {\bf (Hp1)}.
Let  $(\bar S, \bar w^0,\bar w,\bar\alpha,  \bar y^0,\bar  y,\bar\beta)$ be a canonical  local minimizer for \eqref{Pe} such that $\beta(\bar S)<K$,  and let   $(p_0,p, \lambda)$  be a multiplier as in Theorem    \ref{PMPeho}. If {\em  one} of the options {\rm (a)}--{\rm (c)} below is verified, then the process  $(\bar S, \bar w^0,\bar w,\bar\alpha,  \bar y^0,\bar  y,\bar\beta)$   is  fully impulsive. 
\begin{itemize}
\item[{\rm (a)}] For every $s\in[0,\bar S]$, the  $C^0$-Pointwise Rank Condition {\em({\bf I}.1)$_{\bar y(s)}$} is verified.
\item[{\rm (b)}]   For every $s\in[0,\bar S]$,  the $C^0$-Pointwise Rank Condition {\em({\bf I}.2)$_{\bar y(s)}$} is verified, while 
$J :=\big\{s \in [0,\bar S] :\   \hbox{{\em({\bf I}.1)$_{\bar y(s)}$} is not verified}\big\}\ne\emptyset$. Furthermore,  $m_1=m$, and  $\lambda\, \frac{\partial \ell^e}{\partial x}(\bar y(s),\bar w^0(s), \bar w(s),\bar\alpha(s)) =0$ for a.e.  $s\in J$. 
\item[{\rm (c)}] The system is linear, the Kalman Controllability Condition  {\rm ({\bf II})} is verified,  and
$\lambda\, \frac{\partial \ell^e}{\partial x}(\bar y(s),\bar w^0(s), \bar w(s),\bar\alpha(s)) =0$  for a.e. $s\in [0,\bar S]$. 
\end{itemize}
\end{theorem}

Preliminarily, let us   prove the following result:
\begin{lemma}
\label{lmap0} 
Assume {\em (i)} and {\em (ii)} in  hypothesis {\bf (Hp1)},    and let $\pi = 0$.
Then for any subset $\mathcal{J}\subseteq[0,T]$ of positive measure one has neither  
\bel{p01}
p(s) = 0 \ \text{ and } \  
 \ell^e \big(\bar y(s),(\bar y(s), \bar w^0(s),\bar w(s),\alpha(s)\big)  > 0,\quad \text{for a.e.} \,\,s\in \mathcal{J} 
 \eeq
 nor
 \bel{p02}
p(s) = 0  \ \text{ and } \  \frac{\partial \ell^e}{\partial x }\big(\bar y(s), \bar w^0(s),\bar w(s),\alpha(s)\big)\neq  0,  
 \quad  \text{for a.e.}\,\, s\in \mathcal{J}. 
\eeq
\end{lemma}
\begin{proof} By hypothesis {\bf (Hp1)} (i), $\Gamma= \cR\times\hat\Gamma$, with $\hat\Gamma$ a cone of $\cR^n$.     Because of {\bf (Hp1)} (ii) and of the identity $\Gamma^\perp = \{0\}\times \hat\Gamma^\perp$,  the non-transversality condition yields
$
p_0 =-\lambda \frac{\partial \Psi}{\partial t} \big(\bar y^0(\bar S),\bar y(\bar S)\big) + 0 = 0.
$

First, let us assume by contradiction  that 
\eqref{p01} is verified on a subset  $\mathcal{J}\subseteq [0,\bar S]$ of positive measure.
 Since $(p_0,p(s),\pi) = (0,0,0)$ for all  $s\in \mathcal{J}$, by   \eqref{engine} 
 we obtain that 
$
\lambda\ell^e\big(\bar y(s),\bar w^0(s),\bar w(s),\bar \alpha(s)\big) = 0$ for a.e.  $s\in \mathcal{J},
$
which by \eqref{p01}  implies that
$\lambda = 0$.  

  Secondly, assume that  \eqref{p02} is verified on a subset $\mathcal{J}\subseteq [0,\bar S]$ of positive measure. We still have $(p_0, p(s),\pi) = (0,0,0)$ on $\mathcal{J}$ and, by the adjoint equation, we deduce 
$
\lambda\frac{\partial\ell^e}{\partial x}\big(\bar y(s),\bar w^0(s),\bar w(s),\bar \alpha(s)\big)=0 $ for a.e.  $s\in \mathcal{J}$ 
so that by \eqref{p02} one gets again
$\lambda=0$. 

Choose a point $\hat s\in \mathcal{J}$, so that 
 $p(\hat s)=0$. Since in both cases one has
$\lambda = 0$, the adjoint equation is linear in $p$, which in turn  implies that 
$p\equiv 0$ on $[0,\bar S]$.
Therefore,  $(p_0,p ,\pi,\lambda) = 0$, which contradicts the non-triviality condition.
\end{proof} 

\begin{proof}[Proof of Theorem \ref{ThmFully}] Observe that, since $\beta(\bar S)<K$, one has $\pi=0$. 
 
Suppose first that hypothesis (a) is verified. For every $s\in [0,\bar S]$,  by  {({\bf I}.1)$_{\bar y(s)}$} there exist an integer $r\geq 0$ and Lie brackets $B_1,\ldots,B_r\in \B^0$ verifying the rank condition \eqref{r1} and, in view of \eqref{pg0hi}, \eqref{pg000hi},   for all $s\in[0,\bar S]$, one has
$$
p(s)\cdot g_i(\bar y(s))=0, \quad  p(s)\cdot B_j(\bar y(s)) = 0 
$$
 for all  $i=1,\dots,{m_1}$,  $j=1,\dots,r$. Therefore, we obtain $ p(s)=0$ for all  $s\in [0,\bar S].$ 
  Assume  by contradiction that there exists a subset of positive measure $\mathcal{J}\subseteq [0,\bar S]$
 such that  $\bar w^0(s)>0$ for a.e. $s\in \mathcal{J}$. By the positivity of the function $\ell$, this implies that 
  $\ell^e\big(\bar y^0(s),\bar y(s), \bar w^0(s),\bar w(s),\bar \alpha(s)\big)>0$   for a.e. $s\in \mathcal{J}$, which in turn  is ruled out by  Lemma \ref{lmap0} above.
  
Assume now that (b) holds true. If  $s\in [0,\bar S]\backslash J$,  we get  $p(s)=0$ arguing as in the previous case. If $J$ has zero-measure, this also implies that  $p(s)=0$ for all $s\in[0,\bar S]$.  On the contrary, assume that  $J$ has positive measure.   For almost every $s\in J$
   and for  $a:=\bar \alpha(s)$,  by  {({\bf I}.2)$_{\bar y(s)}$} there exist  integers $r,k\geq 0$ and Lie brackets $B_1,\ldots,B_r\in \B^0$, 
  $\hat B_1,\ldots,\hat B_k\in \B^1$ verifying the rank condition \eqref{fc0rank}. Moreover, for almost every  $s\in J$, by  \eqref{pg0hi}, \eqref{pg000hi}, and  \eqref{MP111newpart} one has
\begin{gather*}
p(s)\cdot g_i(\bar y(s))=0, \quad  p(s)\cdot B_j(\bar y(s)) = 0, \\ 
p(s)\cdot [f_{\bar \alpha(s)},g_i](\bar y(s))=0, \quad p(s)\cdot [f_{\bar \alpha(s)},\hat B_l](\bar y(s))= 0,
\end{gather*}
for all $i=1,\dots,{m}$,  $j=1,\dots,r$,  $l=1,\dots,k$.
We then deduce that $p(s)=0$ for almost every $s\in J$. Summing up the above occurrences, by the continuity of $p$ we get  $p(s)=0$ for every $s\in [0,\bar S]$.
 Now assume  by contradiction  that there exists a subset $\mathcal{J}\subseteq[0,\bar{S}]$ of positive measure such that  $\bar w^0(s)>0$ for a.e. $s\in \mathcal{J}$. At this point, the thesis follows arguing exactly as in case (a). 
 
 Finally, suppose that (c) holds true. The linear relations \eqref{c2},  \eqref{hobell}  imply $p\equiv 0$, so, in view of the  hypothesis $\ell>0$ one concludes as in  cases (b) and (c).
\end{proof}

 \begin{remark}{\rm
  As mentioned in the introduction, our conditions might be  regarded as a generalization to impulsive trajectories of   \cite{ChittaroStefani2016}, where one assumes further that  ${\mathcal C}=\cR^m$, the vector fields $g_1,\ldots,g_m$ are of class $C^\infty$,  and their Lie algebra has constant dimension.  In   \cite{ChittaroStefani2016} one considers   
 the (non extended) minimum time problem with  $L^\infty$ controls taking values in an unbounded set. 
Now, since the dynamics is control-affine, an optimal control might fail to exist in this class or even in the class of  $L^1$ functions. On the other hand, if such an optimal control existed, the corresponding cost might or might not coincide with the infimum value of the extended, impulsive problem. Actually, in  \cite{ChittaroStefani2016} one assumes that the optimal process is a {\it normal extremal}, and this is similar to a sufficient condition established in \cite{MottaRampazzoVinter2019} for the avoidance of infimum-gaps. One might conjecture that,  for some reason,\footnote{E.g. because of the {\it abundantness} (see \cite{Warga1972}) of the absolutely continuous  trajectories  in the set of  extended, impulsive trajectories.} a  higher-order Maximum Principle  valid for the impulsive system can be a necessary condition for the non-impulsive unbounded system as well.  }
\end{remark}

\section{Proof of Theorem \ref{PMPeho}}\label{SecProof}
Let $(\bar S, \bar w^0,\bar w,\bar\alpha, \bar y^0,\bar  y,\bar\beta)$ be  a canonical  local minimizer of \eqref{Pe}  verifying $\bar\beta(\bar S)<K$, that we will call the {\it reference process}. Throughout this section $\hat \ell_1(\cdot,0,\cdot)\equiv 0$, as required in the statement of Thm. \ref{PMPeho}. Moreover, we  set
\begin{gather*}
\ds F^e(x,w^0,w,a):=f(x,a)w^0+\sum_{i=1}^mg_i(x)w^i, \quad \text{for all } (x,w^0,w,a)\in\cR^n\times \cR_+\times\C\times A, \\ 
\ds\bar F^e(s):=F_e(\bar y,\bar w^0,\bar w,\bar \alpha)(s), \quad \bar \ell^e(s):=\ell^e(\bar y,\bar w^0,\bar w,\bar \alpha)(s), \quad \text{for a.e. $s\in[0,\bar S]$}.
\end{gather*}
The proof will be divided in several steps. First, following a time-rescaling procedure,  we transform problem  \eqref{Pe} into a problem  on the fixed interval
 $[0,\bar S]$. At this point,  we define two  classes of    variations, comprising  standard {\em needle variations} and {\it  bracket-like variations}, the latter being produced by suitable  {\it instantaneous} perturbations of the reference process.  By using  appropriate powers  of the perturbation parameter $\eps$, all these variations turn out to be of the same order $\eps$. Once this is done, the proof proceeds by some  set-separation arguments. 
 
\subsection{Rescaling the problem}
 
 \begin{definition}Fix $\rho>0$.
 For any   $(S, w^0,w,\alpha,\zeta)\in \mathcal{W}\times L^\infty([0,\bar S], [-\rho,\rho]) ,$  we  say that $(S,w^0,w,\alpha,\zeta,y^0,y,y^\ell,\beta)$ is  a  {\em rescaled (space-time) process} if  $(y^0,y,y^\ell,\beta)$ is  the unique Carath\'eodory solution of 
  \bel{CPR}
  \left\{\begin{split}
   \ds\frac{dy^0}{ds} &=  {w}^0 (1+\zeta),\\
   \ds\frac{dy}{ds} &=F^e(y, {w}^0,{w},\alpha)(1+\zeta),\\
   \ds\frac{dy^\ell}{ds} &= \ell^e(y, {w}^0,{w},\alpha)(1+\zeta),\\
\ds\frac{d\beta}{ds}& = |w| (1+\zeta),\\
  (y^0,& y, y^\ell,\beta)(0) = (0,\xbo,0,0),
  \end{split}
  \right. \qquad \text{a.e. }s\in [0,\bar S],\eeq
and $(S,w^0,w,\alpha,y^0,y,y^\ell,\beta)$ is called {\em  feasible} if $(y^0(S),y(S),\beta(S)) \in \cS\times [0,K]$. 
 \end{definition}
We define  the  {\it rescaled} space-time optimization problem as
\bel{Pee}\tag{P\tiny e}
\left\{\begin{array}{l}
\text{minimize }\Big\{\h((y^0,y)(\bar S)) + y^\ell(\bar S)\Big\}, \\
\text{over feasible rescaled processes 
  $( \bar S,  w^0, w, \alpha, \zeta, y^0, y, y^\ell, \beta)$.}
  \end{array}\right.
\eeq
 It is easy to see  that, for  $\rho>0$ sufficiently small,   the reference  process, regarded as  a process   $(\bar S, \bar w^0,\bar w,\bar \alpha, 0,\bar y^0,\bar y,\bar y^\ell,\bar \beta)$ of  \eqref{CPR},
   is a local minimizer  for \eqref{Pee}, which is a  fixed end-time problem.\footnote{{\em I.e.,} there exists $\delta >0$ such that
$\h((\bar y^0,\bar y)(\bar S)) +\bar y^\ell(\bar S) \leq  
\h((y^0,y)(\bar S)) +y^\ell(\bar S)$
 for all feasible   processes  $( \bar S,  w^0, w, \alpha, \zeta, y^0, y, y^\ell, \beta)$ satisfying
$
{\rm d}\big((\bar S, y^0, y ,y^\ell, \beta) , (\bar S,\bar y^0, \bar y ,\bar y^\ell, \bar \beta)\big)<\delta
$.}
Since  the proof  involves  only space-time trajectories which are  close to  the reference space-time trajectory $(\bar y^0 , \bar y )$,   using standard  truncation and mollification arguments, we can assume the following hypothesis:
\begin{itemize}
\item[{\bf (Hp)$^*$}] {\it all the  assumptions  in {\bf (Hp)} are verified and, moreover,  $\ell^e$,  $f$, the  $g_i$,  their partial derivatives $\ds\frac{\partial \ell^e}{\partial x^j}$, $\ds\frac{\partial f}{\partial x^j}$, $\ds\frac{\partial g_i}{\partial x^j}$ and all the iterated brackets $B\in \B^0$ (as defined in Definition \ref{DB}) are uniformly continuous and bounded. }
\end{itemize}  
Hypothesis {\bf (Hp)$^*$} guarantees that for any $(w^0, w, \alpha, \zeta)\in L^\infty([0,\bar S], W\times A\times \left[-\rho,\rho\right])$ there exists a unique solution $(y^0, y, y^\ell, \beta)$ to \eqref{CPR}, defined on the whole interval $[0,\bar S]$.
Moreover, the input-output map 
\bel{ium}
\Phi:L^\infty\left([0,\bar S], W\times A\times \left[-\rho,\rho\right]\right)\to C^0([0,\bar S],\cR\times\cR^n\times\cR\times\cR),
\eeq
which associates to any control the corresponding solution  to \eqref{CPR}, turns out to be  Lipschitz continuous when  one considers the sup-norm over the set of trajectories, and  the distance  $
\tilde d\big((w^0, w, \alpha, \zeta), (\tilde w^0, \tilde w, \tilde \alpha, \tilde{\zeta})\big):={\rm meas}\,\big\{(w^0,w,\alpha, \zeta)(s)\ne(\tilde w^0,\tilde w,\tilde\alpha,\tilde{ \zeta})(s): \ s\in[0,\bar S]\big\}
$
for every pair  $(w^0, w, \alpha, \zeta)$, $(\tilde w^0, \tilde w, \tilde \alpha, \tilde{\zeta})$ of controls.

\subsection{Needle and bracket-like approximations}\label{ParApproximations}

\begin{definition}[Variation generator]
Let us  define the {\rm  set of   variation generators} as 
 $$
\mathfrak{V}:=\left(W\times A\times \left[-\rho,\rho\right]\right)   \bigcup \mathfrak B^0. \  \footnote{We  recall that $W=\{(w^0,w)\in\R_+\times\C: \ w^0+|w|=1\}$  and $\B^0$ is the set of $C^0$-admissible iterated Lie  brackets of length $\ge2$,  as in Def. \ref{DB}. }
$$
 More specifically, any
    $\cbf=(w^0,w,a,\zeta)\in W\times A\times \left[-\rho,\rho\right]$  will be called a {\rm needle variation generator},
   or  a {\rm  variation generator of length $1$}, while  any bracket  $\cbf=B\in\B^0$ of length $h$ $(\geq 2)$ will be called a 
 {\rm bracket-like variation generator of   length   $h$}.\end{definition}

  To every variation generator $\cbf$ and to each instant $\sb\in(0,\bar S)$, we now associate an infinitesimal space-time variation of the reference trajectory $(\bar y^0,\bar y,{\bar y}^\ell, \bar\beta)$,  whose $y$-component  coincides with either a standard needle variation or  a Lie bracket. 
As usual, the needle variations will be considered at  {\em Lebesgue points} of an appropriate associated function as given in next definition.\footnote{Given $F\in L^1([a,b],\cR^N)$,   $s\in (a,b)$ is called a {\it Lebesgue point}  if  ${\ds\lim_{\delta\to 0}}\frac1\delta \int_{s-\delta}^{s+\delta}|F(\sigma)-F(s)|d\sigma =0$. By the Lebesgue Differentiation Theorem, the set
 of Lebesgue points has   measure $b-a$.} 

\begin{definition}\label{Leb}
We will use $(0,\bar S)_{\rm Leb}$  to denote the full measure  subset of  $(0,\bar S)$ consisting of  the Lebesgue points of the function
$s\mapsto \big(\bar w^0(s),\bar F^e(s),\bar\ell^e(s),|\bar w|(s)\big)$,   $s\in[0,\bar S]$.
\end{definition}

\begin{definition}
 \label{DefVariation}
{\rm (Needle variation)}. For every $\sb\in (0,\bar S)_{\rm Leb}$  and  every needle  variation generator $\cbf=(w^0,w,a,\zeta)$,  consider the  vector \bel{needlevar}\ds
\begin{array}{l}

\left(\begin{matrix}\v^0_{\cbf,\sb}\\ \v_{\cbf,\sb}\\  \v^\ell_{\cbf,\sb}\\\v^{\va}_{\cbf,\sb}\end{matrix}\right) :=
\if{\left(\begin{matrix}w^0\\ F^e(\bar y(\sb),w^0,w^i\\ \ell^e(\bar y(\sb),w^0,w, a)\\|w|\end{matrix}\right)  - \left(\begin{matrix}\ds \frac{d\bar y^0}{ds}(\sb)\\\ds  \frac{d\bar y}{ds}(\sb)\\ \ell^e((\bar y,\bar w^0,\bar w, \bar\alpha)(\sb))\\|\bar w(\sb)|
\end{matrix}\right) =\\\\}\fi
\left(\begin{matrix}w^0(1+\zeta)-\bar w^0(\sb)\\ F^e(\bar y(\sb),w^0,w,a)(1+\zeta)- \bar F^e (\sb)\\ \ell^e(\bar y(\sb),w^0,w, a)(1+\zeta)-  \bar\ell^e(\sb)\\ |w|(1+\zeta) - |\bar w(\sb)|
\end{matrix}\right).
 \end{array}
 \eeq

\noindent {\rm (Bracket-like variation)}. For every $\sb\in (0,\bar S)$ 
and every  bracket-like variation generator $\cbf= B\in \B^0$,  one sets
\bel{bracketvar}
\ds\left(\begin{matrix}\v^0_{\cbf,\sb} \\ \v_{\cbf,\sb} \\  \v^\ell_{\cbf,\sb} \end{matrix}\right) :=  
\left(\begin{matrix} 0 \\ \frac{ B(\bar y(\sb))}{r_{_B}^h} \\  0 \end{matrix}\right),
\eeq
where ${r_{_B}}$  is defined as in Subsection \ref{SecNotation}.
\end{definition}
 
 \begin{definition}[Needle approximation]
\label{DefNeedle}
 Let $\cbf=(w^0, w, a,\zeta)$ be  a   needle variation generator  and let  $\bar s\in (0,\bar S)$.  For any control  $(\tilde w^0,\tilde w,\tilde  \alpha,\tilde{\zeta})$ belonging to the set  $L^\infty\left([0,\bar S], W\times A\times\left[-\rho,\rho\right] \right)$,  the family  $\left\{(\tilde w^0,\tilde w,\tilde  \alpha,\tilde{\zeta})_{\cbf,\sb}^\eps:\ \eps\in(0,\bar s)\right\}$,   defined by 
 \be
\label{ControlNeedle}
(\tilde w^0, \tilde w,\tilde \alpha,\tilde{\zeta})_{\cbf,\sb}^\eps(s):=
\left\{
\begin{split}
&(w^0,w,a,\zeta),\quad &\text{if } s\in [\bar s-\eps,\bar s],\\
&(\tilde w^0,\tilde w,\tilde \alpha,\tilde{\zeta})(s),\quad &\text{if } s\in [0,\bar s-\eps)\cup(\bar s, \bar S],
\end{split}
\right.
\eeq
is called a {\em needle control approximation of  $(\tilde w^0,\tilde w,\tilde  \alpha,\tilde{\zeta})$  at $\sb$ associated to $\cbf$.}
 \end{definition}

In order to state Lemma \ref{LemmaDerSingleNeedle} below  --which is a standard result (see e.g. \cite{Pontryagin}) --, for any $ \tilde y:=(y^0,y,y^\ell,\beta)\in \cR\times\cR^n\times\cR\times\cR$ and any $(w^0,w,a)\in W\times A$, let us set
 $$
\tilde F(\tilde y,w^0,w,a):=( w^0 , 
F^e(y,w^0,w,a) ,  \ell^e(y,w^0,w,a) , |w|)
$$
 and    use  $\tilde M(\cdot,\cdot)$   to denote  the {\it fundamental matrix } of the  {\em  variational equation}
\bel{vareq}
 \ds \frac{d \tilde V}{ds}(s)  =  \frac{\partial \tilde F}{\partial \tilde x} (\bar  y^0(s), \bar  y(s), \bar  y^\ell(s), \bar\beta(s),\bar w^0(s),\bar w(s), \bar\alpha(s))\cdot \tilde V(s) , \quad  \text{a.e. }s\in[0,\bar S].
 \eeq
Namely, for each vector $\tilde v:= \left(v^0,v, v^\ell, v^\va\right)\in\cR^{1+n+1+1}$ and each $s_1\in [0,\bar S]$,  the function
 $\tilde V(\cdot) :=\tilde M(\cdot,s_1)\tilde v$ 
is the solution of \eqref{vareq} with initial condition $\tilde V(s_1)= \left(V^0, V, V^\ell, V^\va\right)(s_1)=\tilde v$.
It is straightforward to check that, for all $s\in [0, \bar S]$ one has:
\begin{itemize}
\item $ \tilde M_{0,j}(s,s_1)=\tilde M_{j,0}(s,s_1) =  \delta_{0,j},\quad \text{for }j=0,\dots,n+2,$
\item $\tilde M_{n+2,j}(s,s_1) =M_{j,n+2}(s,s_1) = \delta_{n+2,j},\quad \text{for }j=0,\dots,n+2,$
\item $ \tilde M_{i,r}(s,s_1) =  M_{i,r}(s,s_1),$ for $i,r=1,\ldots,n,$
\item $ \tilde M_{r,n+1}(s,s_1)=\mu_r(s,s_1):= \ds \int_{s_1}^s\sum_{j=1}^n\frac{\partial \ell^e}{\partial x^j}((\bar  y, \bar  w^0, \bar  w, \bar\alpha)(\sigma))\cdot M_{j,r}(s,\sigma)d\sigma,$ for $r=1,\dots,n,$
\item $  \tilde M_{n+1,n+1}(s,s_1)=1,    $
\end{itemize}
where $M(\cdot,\cdot)$ denotes the fundamental matrix of the state-variational equation in $\cR^n$
\bel{vareqn}
 \ds \frac{d V}{ds}(s)  = \frac{\partial F^e}{\partial x}(\bar  y(s), \bar w^0(s), \bar w(s),\bar\alpha(s))\cdot V(s),   \quad \text{a.e.} \  s\in[0,\bar S].
\eeq
 
\begin{lemma}[Asymptotics of needle variations]
 \label{LemmaDerSingleNeedle}
Assume that $\bar s \in (0,\bar S)_{\rm Leb}$.
For every  needle variation generator $\cbf=(w^0,w,a,\zeta)\in W\times A\times\left[-\rho,\rho\right]$   
 and  for every $s\in(\bar s,\bar S]$,  setting $\mu(s,\bar s):=(\mu_1,\dots, \mu_n)(s,\bar s)$ we get
 \begin{equation}
 \left(\begin{matrix} y^{0\eps}(s)-\bar y^0(s)\\y^\eps(s)-\bar y(s)\\y^{\ell\eps}(s)-\bar y^\ell(s)\\\beta^\eps(s)-\bar \beta(s) \end{matrix}\right) =  \eps \tilde M(s,\bar s)\cdot
 \left(\begin{matrix}\v^0_{\cbf,\sb}\\\v_{\cbf,\sb}\\  \v^\ell_{\cbf,\sb}\\\v^{\va}_{\cbf,\sb}\end{matrix}\right)+ o(\eps)  =  \eps 
 \left(\begin{matrix}\v^0_{\cbf,\sb}\\ M(s,\bar s)\cdot\v_{\cbf,\sb}\\  \mu(s,\bar s)\cdot\v_{\cbf,\sb} + \v^\ell_{\cbf,\sb}\\\v^{\va}_{\cbf,\sb}\end{matrix}\right)+ o(\eps),
\end{equation}
where $\Big(y^{0\eps},y^\eps,y^{\ell\eps}, \beta^\eps\Big)$ denotes the solution of   system \eqref{CPR} corresponding to the needle control approximation $(\bar w^0,\bar w,\bar \alpha,0)_{\cbf,\sb}^\eps$   of  $(\bar w^0,\bar w,\bar  \alpha,0)$  at $\sb$ associated to $\cbf$.
 \end{lemma}

\vsm
Bracket-like approximations, which  can be performed in various ways   (see e.g. \cite{AgrachevSachkov2013,Kno80,BianchiniStefani1999,BianchiniKawski2003,ChittaroStefani2016} and references therein), are here  based on the following result: 
\begin{lemma}\label{applemma}
Assume {\bf (Hp)$^*$} with $\hat\ell_1(\cdot,0,\cdot)\equiv 0$. 
Fix a point $(\tilde{y}^0,\tilde{y},\tilde{y}^\ell,\tilde\beta)\in \cR\times \cR^n\times \cR\times\cR$ and some $a\in A$.
  For every  Lie bracket $B\in\B^0$ of length $h$, there is $\bar\eps>0$  such that,
for any $s\in(0, \bar\eps^{1/h}]$,  there exists a piecewise constant control
 $(\mathpzc{w}_{\cbf,s}^0,\mathpzc{w}_{\cbf,s})$, with 
  $\mathpzc{w}_{\cbf,s}^0(\sigma) = 0$  for all $\sigma\in [0,s]$, $\mathpzc{w}_{\cbf,s}:[0,s]\to\Big\{\pm{\bf e}_1, \dots, \pm{\bf e}_{m_1}\Big\},$   verifying  
\begin{align}
&(\mathpzc{y}^0, \mathpzc{y}^\ell)(\sigma) = (\tilde{y}^0,\tilde{y}^\ell),\qquad  \boldsymbol\beta(\sigma) = \tilde{\beta}+\sigma,\quad \text{for all } \sigma\in [0,s], \label{yo} \\
 &\mathpzc{y}(s)=\tilde y+\left(\frac{s}{r_{_B}}\right)^hB(\tilde y)+o(s^{h}), \label{yt}
\end{align}
where  $r_{_B}$ is the switch-number introduced in Subsect.\ref{SecNotation} and $(\mathpzc{y}^0,\mathpzc{y},\mathpzc{y}^\ell,\boldsymbol\beta)$ denotes  the solution to the space-time control system in \eqref{CPR} corresponding to the control 
$(\mathpzc{w}_{\cbf,s}^0,\mathpzc{w}_{\cbf,s},a,0)$ \footnote{Note that the choice of the element $a$ is irrelevant.}  and the  initial condition $(\mathpzc{y}^0,\mathpzc{y},\mathpzc{y}^\ell,\boldsymbol\beta)(0) =  (\tilde{y}^0,\tilde{y},\tilde{y}^\ell,\tilde\beta)
$.
\end{lemma}

\begin{proof} While the first relation in  \eqref{yo} is trivial, in that  $\mathpzc{w}_{\cbf,s}^0\equiv 0$ and  the Lagrangian 
 $\ell^e(y,a, \mathpzc{w}^0_{\cbf,s},\mathpzc{w}_{\cbf,s}) = \ell_0(y,a)\mathpzc{w}^0_{\cbf,s}\equiv 0,
 $
a proof of \eqref{yt} can be found in  \cite{FeleqiRampazzo2017}. Finally, the second  relation in \eqref{yo} is trivial as well, for $|\mathpzc{w}_{\cbf,s}| \equiv 1$ on $[0,s]$.
\end{proof}

\begin{definition}[Bracket-like approximation]
\label{DefImpulsive}
 Fix   $\sb\in (0,\bar S)$ and let   $\cbf=  B\in \B^0$ be
a  bracket-like variation generator  of length $h$. For each $\eps>0$  such that $\eps<\bar\eps$ and
$ 2\eps^{1/h} < \bar s$, where $\bar\eps$ is as in Lemma \ref{applemma}, 
 consider the dilation
\be
\label{gammaeps} 
\begin{split}
 &\gamma^\eps\colon [\bar {s}-2\eps^{1/h},\sb-\eps^{1/h}]\to [\bar {s}-2\eps^{1/h},\sb] \\
 &\gamma^\eps(\sigma):= (\bar {s}-2\eps^{1/h}) +2\big(\sigma-(\bar {s}-2\eps^{1/h})\big).
\end{split} 
\ee
For any control  $(\tilde w^0,\tilde w,\tilde  \alpha,\tilde\zeta)\in L^\infty\left([0,\bar S], W\times A\times\left[-\rho,\rho\right] \right)$, let us set
\be
\label{ControlB}
(\tilde w^0,\tilde w, \tilde\alpha,\tilde\zeta)_{\cbf,\sb}^\eps (s):=  
\left\{
\begin{split}
&\ds\Big(2\,\tilde w^0,2\,\tilde w ,\tilde \alpha,\tilde\zeta \Big) \circ \gamma^\eps(s),\quad \text{if } s\in [{\sb}-2\eps^{1/h},{\sb}-\eps^{1/h}),\\
&\Big(0,\mathpzc{w}_{\cbf,\eps^{1/h}}(s-(\sb-\eps^{1/h})),a\Big),\quad \text{if } s\in [{\sb}-\eps^{1/h},{\sb}],
 \,\,\\
&\big(\tilde w^0,\tilde w,\tilde \alpha,\tilde\zeta \big)(s),\quad  \text{if } s\in [0,{\sb}-2\eps^{1/h})\cup ({\sb},S],
\end{split}
\right.
\ee
where  $a\in A$  is arbitrary and $\mathpzc{w}_{\cbf,\eps^{1/h}}$ is as in Lemma \ref{applemma}.  
 We refer to  the family of controls $\big\{(\tilde w^0,\tilde w, \tilde  \alpha,\tilde\zeta)_{\cbf,\sb}^\eps:\eps\in(0,\bar\eps),\,2\eps^{1/h} < \bar s\big\}$ as a {\em bracket-like control approximation  of  $(\tilde w^0,\tilde w,\tilde \alpha,\tilde\zeta)$  at $\sb$  associated to $\cbf=B$}. 
 
\end{definition}

\begin{lemma}[Asymptotics   of bracket-like variations]
\label{DerSingleImp}
Let us consider a bracket-like variation generator $\cbf=  B\in \B^0$,  with $B$   of length $h$. For every    point $\sb\in (0,\bar S)$ and  for each $\eps>0$ as in Def. \ref{DefImpulsive},
 let   $(\bar w^0,\bar w,\bar \alpha,0)_{\cbf,\sb}^\eps$ be a bracket-like control approximation  of $(\bar w^0,\bar w,\bar \alpha, 0)$  at $\sb$  associated to $\cbf=B$, and let  $(y^{0\eps},y^\eps, y^{\ell\eps}, \beta^\eps)$ be the corresponding solution of system \eqref{CPR}.
Then,  for every $s\in(\bar s,\bar S]$ one has 
 {\small $$
 \left(\begin{matrix} y^{0\eps}(s)-\bar y^0(s)\\y^\eps(s)-\bar y(s)\\y^{\ell\eps}(s)-\bar y^\ell(s) \end{matrix}\right) =  \eps 
 \left(\begin{matrix}\v^0_{\cbf,\sb}\\ M(s,\bar s)\cdot\v_{\cbf,\sb}\\  \mu(s,\bar s)\cdot\v_{\cbf,\sb} + \v^\ell_{\cbf,\sb} \end{matrix}\right)+ \left(\begin{matrix} 0 \\ o(\eps) \\ o(\eps)  \end{matrix}\right)=\eps \left(\begin{matrix}\ 0\\ M(s,\bar s)\cdot \ds\frac{ B(\bar y(\sb))}{({r_{_B}})^h} \\  \mu(s,\bar s)\cdot \ds \frac{ B(\bar y(\sb))}{({r_{_B}})^h} \end{matrix}\right)+  \left(\begin{matrix} 0 \\ o(\eps) \\ o(\eps)  \end{matrix}\right),
$$}
and
$
\beta^\eps(s) - \bar\beta(s) =\eps^{\frac 1h}.  
$
\end{lemma}

\begin{proof} 
 By the rate-independence of the   control system \eqref{CPR}, 
  $(y^{0\eps},y^\eps, y^{\ell\eps}) =(\bar y^0, \bar y, \bar y^\ell) \circ \gamma^\eps$  on $[\bar s-2\eps^{1/h},\bar s-\eps^{1/h}]$, so that 
$(y^{0\eps},y^\eps, y^{\ell\eps})(\bar s-\eps^{1/h})=(\bar y^0, \bar y, \bar y^\ell)(\bar s).$
Hence
$y^{0\eps}(\bar s)-\bar y^0(\bar s)=0$,  $y^{\ell\eps}(\bar s)-\bar y^\ell(\bar s)=0$,
while
\be
\label{dyepsimpulsive1}
\begin{split}
 \ds y^\eps(\sb)-\bar y(\sb) =&
 \int_{\sb-\eps^{1/h}}^{\sb} \sum_{i=1}^m g_i(y^\eps(s)) \mathpzc{w}_{\cbf,\eps^{1/h}}^i (s-(\sb-\eps^{1/h}))\, \dd s\\  =&\int_{0}^{\eps^{1/h}} \sum_{i=1}^m g_i(y^\eps(s+(\sb-\eps^{1/h})) \mathpzc{w}_{\cbf,\eps^{1/h}}^i (s) \,\dd s,
\end{split}
\ee
where $\mathpzc{w}_{\cbf,\eps^{1/h}}$ is the control associated to the bracket  $B$ as  in Lemma  \ref{applemma}.  
It follows that $y^\eps(\sb)=y^\eps \big( s+(\sb-\eps^{1/h}) \big){\Big|_{s=\eps^{1/h}}} ={\mathpzc{y}^\eps(\eps^{1/h})}$, where we have used $\mathpzc{y}^\eps$   to denote the solution to the Cauchy problem 
 $
\displaystyle\frac{dy}{d\sigma}(\sigma)=\sum_{i=1}^{m} g_i(y(\sigma))\mathpzc{w}^i_{\cbf,\eps^{1/h}}(\sigma)$, $y(0)=\yb(\sb),
 $ 
so that, by  Lemma  \ref{applemma}, we get
$$
  y^\eps(\sb)-\bar y(\sb)-  \left(\frac{\eps^{1/h}}{r_{_B}}\right)^h B(\bar y(\sb)) = 
 {\mathpzc{y}^\eps(\eps^{1/h})}
- \yb(\sb) - \eps\, \frac{ B(\bar y(\sb))}{({r_{_B}})^h} =o(\eps).
$$
 Therefore,
$y^\eps(\sb)-\bar y(\sb)=\eps\, \ds \frac{ B(\bar y(\sb))}{({r_{_B}})^h} + o(\eps)$,  and the proof of the first relation of the thesis is concluded, 
since for every $s\in(\bar s,\bar S]$, the fundamental matrix  $\tilde M(s,\bar s)$ is the differential of the flow map from $\bar s$ to $s$. 
Finally, by the second relation in \eqref{yo}, one has 
$\beta^\eps(s) - \bar\beta(s) =\beta^\eps(\bar s) - \bar\beta(\bar s)
 =  \eps^{\frac 1h}.$ 
\end{proof}
 
\if{ 
\sol{ Stavo pensando se era
veramente necessario stringere il controllo prima fare la impulsive
variation e mi pare di si.
\\
Di fatto, se non stringiamo il controllo ci viene in \eqref{dyepsimpulsive1} il limite
$$
\frac{1}{\eps} \big[ \yb(\sb-\eps^{1/h}) - \yb(\sb) \big],
$$
dove mi viene di usare la Lipschitzianit\`a della traiettoria
$\yb(\cdot)$ per ottenere
$$
\left|\frac{1}{\eps} [ \yb(\sb-\eps^{1/h}) - \yb(\sb) ]\right| \leq K
\frac{\eps^{1/h}}{\eps} \to \infty !!
$$
Siete d'accordo?\\
Mi \`e venuto il dubbio dopo avere parlato con Sussmann, che non vedeva
}}\fi


\subsection{Composition of variations}\label{ParCompositions}
\vsm

Let   $\cbf\in \mathfrak{V}$ be a variation  generator  of length $h\ge1$,  and let $\sb\in (0,\bar S)$.    
For any $\eps>0$ small enough \footnote{Precisely, if $h\ge2$, we require $0<\eps < \bar\eps$, $2\eps^{1/h} < \bar s$, as in Def. \ref{DefImpulsive}, while,  in case $h=1$,  $\eps<\bar s$.}, let us introduce the operator
 $\mathcal{A}_{\cbf,\sb}^\eps: L^\infty\left([0,\bar S],\cR_+\times \C \times A\times\left[-\rho,\rho\right]\right) \to L^\infty\left([0,\bar S],\cR_+\times \C \times A\times\left[-\rho,\rho\right]\right)$  given by
\be
\mathcal{A}_{\cbf,\sb}^\eps  ( w^0,w,\alpha,\zeta):=( w^0, w, \alpha,\zeta)_{\cbf,\sb}^\eps,
\ee 
 \begin{lemma}[Multiple variations at different times]
\label{LemmaMultVariations}
Let   $N>0$  be an integer and  let $\vecc:=(\cbf_1,\dots,\cbf_N)  \in \mathfrak{V}^{N}$ be an $N$-uple of variations of lengths  $\vech:= (h_1,\dots,h_N)\in \cN^{N}$. Fix  $\vecs:=(\sb_1,\dots,\sb_N) \in (0,\bar S)
^N$, where $0=:\bar s_0<\bar s_1<\dots<\bar s_N<\bar S$ and $\sb_j\in (0,\bar S)_{\rm Leb}$ as soon as $h_j = 1$. For each
$\veceps:=(\eps_1,\dots,\eps_N)\in (0,+\infty)^N$ small enough,  let us set
\be
(w^{0\veceps},w^{\veceps} ,\alpha^{\veceps},\zeta^{\veceps}) :=
\A_{\cbf_N,\bar s_{N}}^{\eps_N} \circ \dots  \circ \A_{\cbf_j, \bar s_{j}}^{\eps_j} \circ \dots \circ \A_{\cbf_1, \bar s_{1}}^{\eps_1} (\bar w^0,\bar w,\bar\alpha,0
),
\ee
and let $\big( \bar S,w^{0\veceps},w^{\veceps} ,\alpha^{\veceps},\zeta^{\veceps}, y^{0\veceps},y^{\veceps},y^{\ell\veceps},\beta^{\veceps}\big) $  denote  the corresponding process of \eqref{CPR}.  

Then, for every $s\in(\sb_N,\bar S]$, one has
\be
\label{MainDerFormula}
\left(\begin{matrix}y^{0\veceps}(s)-  \bar y^0(s)\\  y^{\veceps}(s)-\bar y(s)\\ y^{\ell\veceps}(s)-  \bar y^\ell(s)\end{matrix}\right) = \mathlarger{ \mathlarger{ \sum}}_{j=1}^N\,\, \eps_j\left( \begin{matrix} \v^0_{\cbf_j,\sb_j}\\     M(s,\sb_j) \v_{\cbf_j,\sb_j}\\ \mu(s,\sb_j)\cdot\v_{\cbf_j,\sb_j}+\v^\ell_{\cbf_j,\sb_j}\end{matrix} \right)  +o(|\veceps|),
\ee
and
\be
\label{MainDerFormulak}
\beta^{\veceps}(s)-\bar\beta(s)=\sum_{j\in I_1} \eps_{j}\left(
|w_j|(1+\zeta_j) - |\bar w(\bar s_j)|
 \right)  +o(|\veceps|) + \sum_{j\in  \{1,\dots,N\}\setminus I_1} (\eps_{j})^{\frac{1}{h_j}},
\ee 
where   $I_1:=\{j=1,\dots,N:   \ h_j=1\}$.  In particular, if all $\cbf_j$ are needle variations, i.e. $\cbf_j:=(w^0_j,w_j,a_j, \zeta_j)$ for every $j=1,\ldots,N$,
one gets
\be
\label{MainDerFormula1}
\beta^{\veceps}(s)-\bar\beta(s)={\sum_{j=1}^N} \eps_j\big(
|w_j|(1+\zeta_j)  - |\bar w(\bar s_j)|\big)  +o(|\veceps|).
\ee

\if{
where  
\begin{multline}\label{L55}
\Delta_j(\bar s_j):= \\
\left\{
\begin{split}
& f(\bar y(\bar s_j) ,a_j)w^0_j + \sum_{i=1}^m g_{i}(\yb(\bar s_j)) w^{i}_j  -f(\bar y(\bar s_j),\bar \alpha(\bar s_j)) \frac{d\bar\varphi^0}{ds}(\bar s_j) - \sum_{i=1}^m g_{i}(\bar y(\bar s_j))\frac{d\bar \varphi^i}{ds}(\bar s_j),  \ \text{if } j\in I, \\
&\frac{1}{r^{h_j}}B_j(\yb(\bar s_j)), \  \text{if  } j\notin I.
\end{split}
\right.
\end{multline}
Here $r=r(B_j)$ is as in Definition   \ref{DefVariation}.
}\fi
\end{lemma}

\begin{proof}
Let us prove the result by induction on $N,$ the number of composed variations. 
For $N=1,$ the result is proved in  Lemmas \ref{LemmaDerSingleNeedle} and \ref{DerSingleImp}. 
  If  $N\geq 2,$ let us assume  that the result holds true for $N-1$ and let us show that it is valid for $N$ as well.
Let us use $(y^0,y, y^\ell, \beta)^{N}$ and  $(y^0,y,  y^\ell, \beta)^{N-1}$  to denote   the trajectories associated to the $N$ variations and to the first $N-1$ variations, respectively (we  omit  the dependence on $\veceps$ for brevity). Then one has  
\begin{small}
\begin{equation}\label{ind0}
\begin{array}{l}
\left(\begin{matrix}y^{0,N}(\sb_{N})-  \bar y^0(\sb_{N})\\  y^{N}(\sb_{N})-\bar y(\sb_{N})\\ y^{\ell,N}(\sb_{N})-  \bar y^\ell(\sb_{N}) \\\beta^N(\sb_{N})-\bar\beta(\sb_{N})\end{matrix}\right) =\left(\begin{matrix}y^{0,N}(\sb_{N})-  y^{0,(N-1)}(\sb_{N})\\  y^{N}(\sb_{N})-y^{N-1}(\sb_{N})\\ y^{\ell,N}(\sb_{N})-  y^{\ell,(N-1)}(\sb_{N}) \\ \beta^N(\sb_{N})-\beta^{N-1}(\sb_{N}) \end{matrix}\right) +
\left(\begin{matrix}y^{0,(N-1)}(\sb_{N})- \bar y^0(\sb_{N})\\ y^{N-1}(\sb_{N})- \bar y(\sb_{N})\\ y^{\ell,(N-1)}(\sb_{N}) -  \bar y^\ell(\sb_{N})\\ \beta^{N-1}(\sb_{N}) -\bar\beta(\sb_{N}) \end{matrix}\right).   
\end{array}
\end{equation}
\end{small}
By  the inductive hypothesis, we get that
\bel{ind}
\left(\begin{matrix}y^{0,(N-1)}(\sb_{N})-  \bar y^0(\sb_{N})\\  y^{N-1}(\sb_{N})-\bar y(\sb_{N})\\ y^{\ell,(N-1)}(\sb_{N})-  \bar y^\ell(\sb_{N})\end{matrix}\right) = \mathlarger{ \sum }_{j=1}^{N-1}\,\, \eps_j\left( \begin{matrix} \v^0_{\cbf_j,\sb_j}\\     M(\sb_{N},\sb_j) \v_{\cbf_j,\sb_j}\\ \mu (\sb_{N},\sb_j) \v_{\cbf_j,\sb_j}+\v^\ell_{\cbf_j,\sb_j}\end{matrix} \right)  +o(|(\eps_1,\dots,\eps_{N-1})|),
\eeq
and, setting  $I^{N-1}_1:=\{j=1,\dots,N-1:   \ h_j=1\}$, 
\bel{indb}\begin{array}{l}
\ds\beta^{N-1} (\sb_{N})-\bar\beta (\sb_{N})
=\sum_{j\in I^{N-1}_1} \eps_{j}\left(
|w_j|(1+\zeta_j) - |\bar w (\sb_{j})|
 \right) \\\qquad\qquad\qquad\qquad\qquad \ds+o(|(\eps_1,\dots,\eps_{N-1})|) + \sum_{j\in \{1,\dots,N-1\}\setminus I^{N-1}_1}(\eps_{j})^{\frac{1}{h_j}}. \end{array}
\eeq
We  claim that
\bel{th}
\left(\begin{matrix}y^{0,N}(\sb_{N})-  y^{0,(N-1)}(\sb_{N})\\  y^{N}(\sb_{N})-y^{N-1}(\sb_{N})\\ y^{\ell,N}(\sb_{N})-  y^{\ell,(N-1)}(\sb_{N}) \end{matrix}\right)  = \eps_N
 \left(\begin{matrix}\v^0_{\cbf_N,\sb_N}\\ \v_{\cbf_N,\sb_N}\\  \v^\ell_{\cbf_N,\sb_N}\end{matrix}\right)+o(|\veceps|),
\eeq
and
\bel{th2}
\beta^N(\sb_{N})-\beta^{N-1}(\sb_{N})=\left\{\begin{array}{l}  \eps_{N}\big(|w_N|(1+\zeta_N) - |\bar w (\sb_{N})|\big)  +o(|\eps_{N}|),  \quad\text{if $h_N=1$,} \\ \ds(\eps_{N})^{\frac{1}{h_N}},  \qquad\text{if $h_N\ge 2$}.
 \end{array}\right.
\eeq 
 Once one has  proven the claim,  the validity of \eqref{MainDerFormula} and \eqref{MainDerFormulak} follows easily by \eqref{ind0}-\eqref{indb},  by the properties of the fundamental matrix  $\tilde M(s,\bar s_N)$.
To prove \eqref{th},  \eqref{th2}  we first consider the case when the length $h_N$ of   the $N$th variation $\cbf_N$  is  $\geq 2$.
 
    {\em Case $h_N\ge2$.}   Here  $\cbf_N=B_N\in\B^0$  is a bracket-like variation and  one has
$$
\left(\begin{matrix}y^{0,N}\\  y^N\\ y^{\ell,N} \\ \beta^N \end{matrix}\right)( \sb_{N}-\eps_{N}^{1/h_{N}})=\left(\begin{matrix}y^{0,N-1}\\  y^{N-1}\\ y^{\ell,N-1} \\ \beta^{N-1} \end{matrix}\right)(\sb_{N}),
$$
so that  
$y^{0,N}(\sb_{N})-  y^{0,N-1}(\sb_{N})=0$,  $y^{\ell,N}(\sb_{N})-  y^{\ell,N-1}(\sb_{N})=0$,  $\beta^N(\sb_{N})- \beta^{N-1}(\sb_{N})=\eps_{N}^{1/h_{N}}$,
while
 \begin{equation*}
\begin{split}
 y ^{N}  (\sb_{N}) -    y ^{N-1}(\sb_{N})  =& \int_{\sb_{N}-\eps_{N}^{1/h_{N}}}^{\sb_{N}}  \sum g_i(y^{N}(s)) \mathpzc{w}_{\cbf_N,\eps_{N}^{1/h_{N}}}^i \Big(s-\big( \sb_{N}-\eps_{N}^{1/h_{N}} \big) \Big)  \dd s  \\
 =&\int_{0}^{\eps^{1/h}} \sum_{i=1}^m g_i(y^N(s+(\sb_{N}-\eps_N^{1/h_N})) \mathpzc{w}_{\cbf_N,\eps_N^{1/h_N}}^i (s) \,\dd s,
\end{split}
\end{equation*}
where the control $ \mathpzc{w}_{\cbf_N,\eps_{N}^{1/h_{N}}} $ is as in Lemma  \ref{applemma}.
 If   $\mathpzc{y}^N$    denotes the solution to the Cauchy problem 
 $
\displaystyle\frac{dy}{d\sigma}(\sigma)=\sum_{i=1}^{m} g_i(y(\sigma)) \,\mathpzc{w}^i_{\cbf_N,\eps_N^{1/h_N}}(\sigma)$,  $y(0)= y ^{N-1}(\sb_{N}),$ 
then $y^N(\sb_{N})=y^N\big( s+(\sb_{N}-\eps_N^{1/h_N}) \big){\Big|_{s=\eps_N^{1/h_N}}} ={\mathpzc{y}^N(\eps_N^{1/h_N})}$ and, by  Lemma  \ref{applemma}, we get
 \begin{equation*}
\begin{split}
  y^N( \sb_{N})- y^{N-1}(\sb_{N})-   &\left(\frac{\eps_N^{1/h_N}}{r_{_{B_N}}}\right)^{h_N} B_N(\bar y(\sb_N))  \\
  &={\mathpzc{y}^N(\eps_N^{1/h_N})}
- y^{N-1}(\sb_{N}) -  \frac{\eps_N  }{(r_{_{B_N}})^{h_N}}  B_N(y^{N-1}(\sb_{N}))  \\
&\qquad+  \frac{\eps_N  }{(r_{_{B_N}})^{h_N}}  B_N(y^{N-1}(\sb_{N})) - \frac{\eps_N  }{(r_{_{B_N}})^{h_N}}  B_N(\bar y(\sb_N)) \\ 
 &= o(\eps_{N})+   \frac{\eps_N  }{(r_{_{B_N}})^{h_N}}  B_N(y^{N-1}(\sb_{N})) - \frac{\eps_N  }{(r_{_{B_N}})^{h_N}}  B_N(\bar y(\sb_N)).
\end{split}
\end{equation*}
Now by  the continuity of $B_N$ and the inductive hypothesis \eqref{ind}, it follows that
\begin{multline*}
\displaystyle \left| \frac{\eps_N  }{(r_{_{B_N}})^{h_N}}  B_N(y^{N-1}(\sb_{N})) -   \frac{\eps_N  }{(r_{_{B_N}})^{h_N}}  B_N(\bar y(\sb_N)) \right| \\
  \ds\le \frac{\eps_N  }{(r_{_{B_N}})^{h_N}}\,\omega_{B_N}\left(|  y^{N-1}(\sb_{N})-  \bar y(\sb_N)|\right)  \le   \frac{\eps_N  }{(r_{_{B_N}})^{h_N}}\,\omega_{B_N}\left(\mathcal{O}\,(\eps_1+\dots+\eps_{N-1})\right),   
\end{multline*}
where $\omega_{B_N}$ denotes the modulus of continuity of ${B_N}$ and we use  $\mathcal{O}$ to mean a nonnegative function such that $\mathcal{O}(r)\leq Cr$ for all $r\geq 0$, for some constant $C>0$. Therefore, 
$
 \ds y^N( \sb_{N})- y^{N-1}(\sb_{N})= \frac{\eps_N  }{(r_{_{B_N}})^{h_N}}  B_N(\bar y(\sb_{N})) +o(|\veceps|),
  $
which concludes the proof in this  case.

 {\em Case $h_N=1$.}  Here 
   $\cbf_N=(w^0_N,w_N,a_N,\zeta_N)$  and the aimed estimate is rather standard. Nonetheless, we perform it for the sake of  self-consistency. One has
\begin{multline*}
\ds y^{N}(\sb_{N}) -  y^{N-1}(\sb_{N})= 
\int_{\sb_{N}-\eps_{N}}^{\sb_{N}}   \left[F^e(y^{N}(s),w^0_{N}, w_N, a_N)(1+\zeta_N)  \right. \\
\ds \left. - F^e(y^{N-1}(s),\wb^0(s),  \wb(s),\bar\alpha(s))\right]\dd s  = \int_{\sb_{N}-\eps_{N}}^{\sb_{N}}  \big( r_1(s) + r_2 + r_3(s)\big) \dd s,
\end{multline*}
where 
$$
\begin{array}{l}
\ds r_1(s) :=    F^e(y^{N}(s),w^0_{N}, w_{N}, a_{N})(1+\zeta_N)  
 - F^e(\yb(\sb_{N}),w^0_{N}, w_{N}, a_{N})(1+\zeta_N),
\\ [1.5ex]
\ds r_2 :=    F^e(\yb(\sb_{N}),w^0_{N}, w_{N}, a_{N})(1+\zeta_N)
-\bar F^e(\sb_{N}),
\\  [1.5ex]
\ds r_3(s) :=  \bar F^e(\sb_{N})
-F^e(y^{N-1}(s),\wb^0(\sb_{N}), \wb(\sb_{N}),\bar\alpha(\sb_{N})).
\end{array}
$$
Let us start by estimating $r_1.$ Observe that, for $s\in [\sb_{N}-\eps_{N},\sb_{N}],$ 
$$
|y^{N} (s) - \yb(\sb_{N})| \leq |y^{N} (s) - y^{N-1}(s)| +|y^{N-1} (s) - \yb(s)|+  |\yb(s) - \yb(\sb_{N})|.
$$
Moreover,   on $[\sb_{N}-\eps_{N},\sb_{N}]$,  one has  
$\|y^{N}   -y^{N-1}\|_\infty =\mathcal{O} \,(\eps_N)$  
by the Lipschitz continuity of the input-output map $\Phi$ defined in \eqref{ium}; 
$
 \|y^{N-1}  - \yb \|_\infty =\mathcal{O} \,(\eps_1+\dots+\eps_{N-1})  
$ 
  by the inductive hypothesis \eqref{ind}; and
 $ \|\yb(s) - \yb(\sb_{N})\|_\infty= \mathcal{O} \,(\eps_N)$ 
 by the Lipschitz continuity of the reference trajectory.
 Hence
 $\|y^{N} (s) - \yb(\sb_{N})\|_\infty = \mathcal{O}(|\veceps|)$, 
so that
$$
\left|  \int_{\sb_{N}-\eps_{N}}^{\sb_{N}} r_1(s) \dd s\right| 
 \leq \int_{\sb_{N}-\eps_{N}}^{\sb_{N}}L\, |y^{N}(s) - \yb(\sb_{N})| \dd s  =\eps_{N}\, \mathcal{O}(|\veceps|), 
$$
where $L$ is a suitable  positive constant. By the previous estimates  and recalling that  $\sb_{N}$ is a Lebesgue point of the map  in Def. \ref{Leb},  we get
\begin{equation*}
\begin{split}
\ds\Big|  \int_{\sb_{N}-\eps_{N}}^{\sb_{N}}     r_3(s)  \dd s \Big|  \leq  & \Big|  \int_{\sb_{N}-\eps_{N}}^{\sb_{N}}  \big[ \bar F^e(\sb_{N})- \bar F^e(s)\big]\,ds\Big|  \\
& \ds+\Big| \int_{\sb_{N}-\eps_{N}}^{\sb_{N}}  \big[ \bar F^e(s)- F^e(y^{N-1}(s),\wb^0(s), \wb(s), \bar\alpha(s)) \big]\,ds\Big|  \\
& \ds\leq   o(\eps_N)+\eps_N\,\mathcal{O} (|(\eps_1,\dots,\eps_{N-1})|).
\end{split}
\end{equation*}
 Therefore,  
 $
 y^{N}(\sb_{N}) -  y^{N-1}(\sb_{N}) 
= \eps_{N}\, r_2  + o(|\veceps|),
 $
and  the relation in  \eqref{th} concerning the state variables is proven. The proofs  of the other relations are  similar and actually easier, so we omit them. 
\end{proof}

\subsection{Set separation}\label{ParReachable}
Given a   process $( \bar S,  w^0, w, \alpha, \zeta, y^0, y, y^\ell, \beta)$ of the rescaled problem \eqref{Pee},  
 let us introduce the total cost component 
\bel{nuova}
 y^c(s) := \h(y^0(s),y(s)) + y^\ell(s), \quad  s\in [0,\bar S]. \   \footnote{ The function $y^c$ can be obviously  regarded  as the solution of  \newline 
 $\frac{dy^c}{ds}  =\Big(\frac{\partial \h}{\partial t} w^0 + 
\frac{\partial \h}{\partial x} \left(f( y,\alpha) w^0+ \sum_{i=1}^{m}g_{i}(y)w^i\right)+\ell^e(y,w^0,w,\alpha)\Big)(1+\zeta)$, \,   
$y^c(0) =  \h(0,\xbo,0).$}
 \eeq
 
Setting $ \bar y^c(s) := \h(\bar y^0(s), \bar y(s)) + \bar y^\ell(s)$, $s\in [0,\bar S]$,  for any $\delta>0$ we  define the {\em $\delta$-reachable set}  $\mathscr{R}_\delta$ and its {\em projection} $\mathscr{R'}_\delta$ as
$$
\begin{array}{l}
\mathscr{R}_\delta:=
\left\{\begin{array}{c}\displaystyle\big(y^0,y,y^c,\beta\big)(\bar S): \,  (S, w^0,w, \alpha,\zeta,y^0,y,y^\ell, \beta)  \ \text{verifies}\\  
 \displaystyle \d\left((\bar S,y^0,y, y^c,\beta),(\bar S, \bar y^0,\bar y,\bar y^c,\bar\beta)\right)<\delta
\end{array}
\right\}\subseteq\cR^{1+n+1+1},
\\ [1.9ex]
{\mathscr{R}'}_\delta:= \Big\{(y^0,y,y^c)(\bar S): \, (y^0,y,y^c,\beta)(\bar S)\in \mathscr{R}_\delta\Big\}\subseteq\cR^{1+n+1}.
 \end{array}$$	 
When {\it  all $\cbf_j=(w^{0}_j,w_j,a_j,  \zeta_j)$ ,   $j=1,\dots N$,  are needle variations, } we define the set
$$
E:= \left\{\begin{array}{c}\left(\begin{matrix}  \v^0_{\cbf_j,\sb_j}\\   M(\bar{S},\sb_j)\cdot \v_{\cbf_j,\sb_j}\\ \frac{\partial\bar \Psi}{\partial t}((\bar S)\v^0_{\cbf_j,\sb_j} + \frac{\partial\bar \Psi}{\partial x}(\bar S)\cdot M(\bar{S},\sb_j)\cdot\v_{\cbf_j,\sb_j}+  \mu(\bar{S},\sb_j)\cdot\v_{\cbf_j,\sb_j}+\v^\ell_{\cbf_j,\sb_j}\\ |w| (1+\zeta_j)-|\bar w(\sb_j)| \end{matrix}\right), \\ j=1,\dots N\end{array}\right\}
$$
where $\ds\frac{\partial\bar \Psi}{\partial t}(\bar S):= \frac{\partial\Psi}{\partial t}((\bar y^0,\bar y)(\bar S))$, $\frac{\partial\bar \Psi}{\partial x}(\bar S):= \frac{\partial\Psi}{\partial x}((\bar y^0,\bar y)(\bar S))$,   and its  projection $E'$, 
$$
E':=  \left\{\begin{array}{c}\left(\begin{matrix}  \v^0_{\cbf_j,\sb_j}\\   M(\bar{S},\sb_j)\cdot \v_{\cbf_j,\sb_j}\\ \frac{\partial\bar\Psi}{\partial t}((\bar S)\v^0_{\cbf_j,\sb_j} + \frac{\partial\bar\Psi}{\partial x}(\bar S)\cdot M(\bar{S},\sb_j)\cdot\v_{\cbf_j,\sb_j}+  \mu(\bar{S},\sb_j)\cdot\v_{\cbf_j,\sb_j}+\v^\ell_{\cbf_j,\sb_j}  \end{matrix}\right),  \\ j=1,\dots N\end{array}\right\}.
$$
Finally, let us define the convex cones
\bel{RR'}
{R}:={\rm span}^+ (E)\subset \cR^{1+n+1+1},\quad {R'}:={\rm span}^+ (E')\subset \cR^{1+n+1},
\ee
where, for a given subset  $\Theta$ of a vector space,   ${\rm span}^+(\Theta)$    denotes its {\em positive span}.

  \begin{lemma}
 {\rm (i)} The set $R'$ is a Boltyanski approximating cone of the set ${\mathscr{R'}_\delta}$  at the point $(\bar y^0,\bar y,\bar y^c)(\bar S)$.

{\rm(ii) }When  all $\cbf_j=(w^{0}_j,w_j,a_j,\zeta_j)$, for $j=1,\dots, N$,  are needle variations,
 the set  ${R}$ is  a Boltyanski approximating cone of the set  ${\mathscr{R}}$  at $(\bar y^0,\bar y,\bar y^c,\bar \beta)(\bar S)$.
  \end{lemma}
  
  \begin{proof}
Let us set  $ y^{c\veceps}(s):=   \h((y^{0\veceps},y^{\veceps})(s))+y^{\ell\veceps}(s) $, where $y^{\ell\veceps}$, $y^{0\veceps}$, and $y^{\veceps}$ are as in Lemma \ref{LemmaMultVariations}. 
  By \eqref{MainDerFormula} we get 
\begin{multline*}
y^{c\veceps}(\bar S) - \bar y^c(\bar S) = 
\sum_{j=1}^N\,\, \eps_j\Big(\ds \frac{\partial\bar\h}{\partial t}(\bar S)\v^0_{\cbf_j,\sb_j} + \frac{\partial\bar \h}{\partial x}(\bar S) \cdot M(\bar{S},\sb_j)\cdot\v_{\cbf_j,\sb_j}  \\
+  \mu(\bar{S},\sb_j)\cdot\v_{\cbf_j,\sb_j}+\v^\ell_{\cbf_j,\sb_j} \Big)+o(|\veceps|).
\end{multline*}
 Therefore, part (ii) of the statement follows from Lemma \ref{LemmaMultVariations}.

To prove part (i), for some $\tilde\eps>0$ sufficiently small, let us define the function  $F:(0,+\infty)^N\cap\tilde\eps\,\Ba_N\to \cR^{1+n+2}$ by setting 
   $
  F(\veceps) := \left(\begin{matrix}  y^{0\veceps}(\bar S),  y^{\veceps}(\bar S), y^{c\veceps}(\bar S)\end{matrix}\right).
  $
It is straightforward to prove that 
   $
 F(\veceps) =  \left(y^{0}(\bar S), y(\bar S), y^{c}(\bar S))\right)  + L\cdot \veceps + o(|\veceps|),
 $
 where  the linear operator $L\in {\rm Hom}(\cR^N,\cR^{1+n+1})$ is   defined by 
$$
L\cdot \veceps:=\ds \sum_{j=1}^N\,\, \eps_j\left( \begin{matrix} \v^0_{\cbf_j,\sb_j}\\     M(\bar S,\sb_j) \v_{\cbf_j,\sb_j}
\\ \frac{\partial\bar\Psi}{\partial t}(\bar S)\v^0_{\cbf_j,\sb_j} + \frac{\partial\bar\Psi}{\partial x}(\bar S) \cdot M(\bar{S},\sb_j)\cdot\v_{\cbf_j,\sb_j}+ \mu(\bar{S},\sb_j)\cdot\v_{\cbf_j,\sb_j}+\v^\ell_{\cbf_j,\sb_j}\end{matrix}\right).
$$
Hence (i) is proved, in that $R'=L \cdot (0,+\infty)^N$.
\end{proof}
 
Let us consider the {\it profitable set} $\mathscr{P}$ and its  projection  $\mathscr{P}'$, defined as
$$
\begin{array}{l}
  \mathscr{P} :=   \cS\times \left(-\infty,\bar y^c(\bar S)\right)\times
 [0,K]\, \bigcup \,\left\{(\bar y^0,\bar y,\bar y^c,\bar \beta)(\bar S) \right\}, \\ [1.5ex]
 \mathscr{P}'
:=  \cS\times \left(-\infty,\bar y^c(\bar S)\right)\, \bigcup \, \left\{(\bar y^0,\bar y,\bar y^c)(\bar S) \right\},
\end{array}
 $$
and let $\Gamma$ be a  Boltyanski approximating cone for the target $\cS$ at $(\bar y^0,\bar y)(\bar S)$. 
Recalling that  $\bar \beta(\bar S)<K$,
 one trivially checks that the sets
  $$
 {P} := \Gamma \times\cR_-\times\{0\}, \qquad {P}' := \Gamma \times \cR_-,
  $$
are  Boltyanski approximating cones of $\mathscr{P}$ at $(\bar y^0,\bar y, \bar y^c, \bar\beta)(\bar S)$ and of $\mathscr{P'}$ at $(\bar y^0,\bar y, \bar y^c)(\bar S)$, respectively.  
We will need the following elementary result: 
\begin{lemma}
\label{separazione}
 There exists $\delta>0$ such that  
the sets $\mathscr{P}'$ and ${\mathscr{R}'}_\delta$   are locally separated  at $(\bar y^0,\bar y,\bar y^c)(\bar S)$.
\end{lemma}

\begin{proof} Suppose by contradiction  that  for every $\delta>0$ the sets $\mathscr{P}'$ and ${\mathscr{R}'}_\delta$  are not locally separated  at $(\bar y^0,\bar y,\bar y^c)(\bar S)$.  Then, given $\delta\in(0, K-\bar\beta(\bar S))$,\footnote{This interval is not empty, for $\bar\beta(\bar S)<K$.} there exists a process $(\bar S,w^0,w,\alpha,\zeta, y^0,y,y^\ell, \beta)$  of \eqref{CPR}  verifying 
$$
\ds (y^0,y,y^c)(\bar S)\in {\mathscr{R}'}_\delta\cap \mathscr{P}', \quad {\rm d}((y^0,y,y^c, \beta),(\bar y^0,\bar y,\bar y^c,\bar \beta))<\delta.
$$
This implies that $\beta(\bar S) \leq \delta + \bar \beta(\bar S)<K$, thus the final point $(y^0,y,y^c,\beta)(\bar S)\in\mathscr{R}_\delta\cap \mathscr{P}$. Hence,  for every $\delta\in(0, K-\bar\beta(\bar S))$  the sets  
$\mathscr{P}$ and $\mathscr{R}_\delta$ are not locally separated, which contradicts  the local optimality of the reference process.
\end{proof}
 
By Lemma \ref{separazione} the projected reachable set $\mathscr{R}'_\delta$  is locally separated from the projected profitable set  $\mathscr{P'}$ at $(\bar y^0,
 \bar y, \bar y^c)(\bar S)$, for some $\delta>0$. Therefore, since $R'$ and $P'$ are    approximating cones to $\mathscr{R}'_\delta$ and $\mathscr{P'}$, respectively,  and $P'$ is not a subspace,   in view of Theorem \ref{ThmLocallySep} there exists   a vector $(\xi_0,\xi, \xi_c)\in\cR^{1+n+1}$ verifying
$$
0\neq(\xi_0,\xi, \xi_c)\in {R'}^{\bot}\cap(-{P'}^{\bot}).
$$
Since  ${P'}^{\bot} =\Gamma^\bot\times\cR_+,
$  
one gets
$(\xi_0,\xi)\in -\Gamma^\bot,$  $ \xi_c = -\lambda \leq 0,$
and
$$ \xi_0\v^0+\xi\cdot\v+\xi_c\v^c \leq 0 \qquad \forall (\v^0,\v,\v^c)\in R'.$$
By the definition of $R'$ given in \eqref{RR'},
the latter relation is verified if and only if 
\begin{multline*}
\Big(\xi_0-\lambda\frac{\partial\bar\Psi}{\partial t}(\bar S)\Big)\v^0_{\cbf_j,\sb_j}+\Big(\xi-\lambda \frac{\partial\bar \Psi}{\partial x}(\bar S)\Big)\cdot M(\bar{S},\sb_j)  \cdot \v_{\cbf_j,\sb_j}\\
-\lambda\left( \mu(\bar{S},\sb_j)\cdot\v_{\cbf_j,\sb_j}+\v^\ell_{\cbf_j,\sb_j}\right)  \leq 0,\quad \text{for all } j=1,\ldots,N.
\end{multline*}
Therefore, setting 
$$
(p_0,p)(s):= \left(\xi_0   -\lambda\frac{\partial\bar\Psi}{\partial t}(\bar S)  \,\,,\,\,\left(\xi-\lambda \frac{\partial\bar \Psi}{\partial x}(\bar S)\right)\cdot M(\bar S,s) -\lambda\mu(\bar{S},s)
\right),$$
we obtain  that  the multiplier  $(p_0, p,\lambda)\in \cR\times AC\left([0,\bar S],\cR^{ n}\right)
\times \cR_+$ verifies
\bel{quasimax}
p_0\v^0_{\cbf_j,\sb_j}+p(\sb_j)\cdot \v_{\cbf_j,\sb_j}-\lambda\v^\ell_{\cbf_j,\sb_j}  \leq 0,\quad \text{for every }j=1,\dots,N,
\eeq 
 the  non-triviality  condition \eqref{fe1}, 
 and (by $M(\bar S,\bar S) = {\rm Id}$, $\mu(\bar{S},\bar S)=0$) the non-transversality
condition \eqref{fe4}.
Moreover, by the definitions of $ M(\bar S,\cdot)$ and  $\mu(\bar S,\cdot)$,  the path $p$ solves the adjoint equation \eqref{fe2}. Finally,  for a needle variation generator $\cbf_j=(w^0_j,w_j,a_j,\zeta_j)$, by \eqref{quasimax} we get 
$$
\begin{array}{l}
H\Big(\bar y(\sb_j),p_0,p(\sb_j),0,\lambda,w^0_j(1+\zeta_j), w_j(1+\zeta_j), a_j\Big) \\ [1.5ex]
\qquad\qquad\qquad\qquad\qquad -H\Big(\bar y(\sb_j),p_0,p(\sb_j),0,\lambda,\bar w^0(\sb_j),\bar w(\sb_j),\bar\alpha(\sb_j)\Big)\leq 0,
\end{array}
$$
while, for  a bracket-like variation generator $\cbf_j=B_j$, 
 we obtain
$p(\sb_j)\cdot B_j (\bar y(\sb_j)) \leq 0$.


  \subsection{Conclusion of the proof}\label{ParConclusion}

\vsm
 To conclude the proof  we need to extend the previous inequalities to almost all $s\in [0,\bar S]$ and all variations generators $
 \mathbf{c}\in \mathfrak{V}$. This will be achieved   via density arguments coupled with infinite intersection criteria.  Though this is a quite   standard procedure, we give the details for the sake of completeness.  
By Lusin's Theorem, one has that  
$
\ds (0,\bar S)_{\rm Leb}=\bigcup_{k=0}^{+\infty} E_k,
$
where $E_0$ has null measure and, for every $k\in\cN$, the set  $E_k$ is   compact  and the restriction to $E_k$ of the measurable  map considered in Definition \ref{Leb}  is continuous.  For every $k$, let $D_k\subseteq E_k$ be the set of density points\footnote{ We recall that  $t\in \tilde E\subset\cR$  is a {\em density point} for $\tilde E$ if $\ds \lim_{\delta\to0^+}\frac{{\rm meas} \big([t-\delta,t+\delta]\cap \tilde E\big)}{2\delta}=1$.} of $E_k.$
Since $D_k$ and $E_k$ have the same measure, by the Lebesgue density Theorem, $\ds D=\bigcup_{k=0}^{+\infty} D_k\subset [0,\bar S]$ has full measure.

\begin{definition} Let   $F$ be an arbitrary subset  of $D\times  \mathfrak{V}$.  We say that a triple $(\bar p_0,\bar p, \lambda)\in \cR^{1+n+1}$  {\rm  verifies  (P)$_{F}$} if $\lambda\ge0$ and, setting $p_0:=\bar p_0$,   $p(\cdot):=\bar p\cdot M(\bar S,\cdot)$, one has that:

{\rm (i)} \,  $\ds( p_0, p(\bar S))+ \lambda\left(\frac{\partial\bar\Psi}{\partial t}(\bar S)\,,\, \frac{\partial\bar\Psi}{\partial x} (\bar S)\right)\in -\Gamma^{\bot}$;

{\rm (ii) }\, for every $(s,\cbf)\in F$ with $\cbf=(w^0,w,a,\zeta)$, the following inequality \newline
$H\Big(\bar y(s),p_0,p(s),0,\lambda,w^0(1+\zeta), w(1+\zeta),a\Big)\le H\Big(\bar y(s),p_0,p(s),0,\lambda,\bar w^0(s),\bar w(s),\bar\alpha(s)\Big)$
holds true, while  for every $(s,\cbf)\in F$ such that  $\cbf=B\in \B^0$,
$p(s)\cdot B(\bar y(s))\leq 0$.
\end{definition}
For any given  subset  $F\subset D\times  \mathfrak{V}$, let us set
$$
 \Lambda(F):=\Big\{(\bar p_0,\bar p, \lambda)\in\cR^{1+n+1}: \,  |(\bar p_0,\bar p,\lambda)|=1, \, (\bar p_0,\bar p, \lambda) \text{ verifies (P)$_{F}$}\Big\}.
$$
Our goal consists in showing  that $\Lambda(F)\neq \emptyset$ for some $F$ comprising 
pairs $(s,\mathbf{c)},$ such that the union of all times $s$ is a full measure subset of $[0,\bar S]$ and $\mathbf{c}$ can range over all  $\mathfrak{V}$. Clearly, for arbitrary subsets   $F_1$, $F_2$  of $D\times  \mathfrak{V}$ the sets 
$\Lambda(F_1), \Lambda(F_2)$, if not empty,  are  compact and $\Lambda(F_1\cup F_2)=\Lambda(F_1)\cap \Lambda(F_2)$.
By  the previous step,   $\Lambda(F)\ne\emptyset$ as soon as $F$ is {\it finite}  and of the form
\be\label{Fform}
\Big\{(\sb_1,\cbf_1), \dots, (\sb_N,\cbf_N)\Big\}, \quad  {\rm with }\,\, 0=:\bar s_0<\bar s_1<\dots<\bar s_N<\bar S.
\ee
In order to prove that $\Lambda(F)\ne\emptyset$ for an arbitrary {\it finite} set $F\subset D\times  \mathfrak{V}$,  we have  to show that it is non-empty even when 
$F=\Big\{(\sb_1,\cbf_1), \dots, (\sb_N,\cbf_N)\Big\}$ with $0=:\bar s_0\le\bar s_1\le\dots\le\bar s_N<\bar S$ and one allows that  $\bar s_j =\bar s_{j+1}$ for some $j=0,\ldots,N-1$. To this end, observe that every $\bar s_j$ belongs to some  set of density points $D_k$,  that we denote $D_{k(j)}$. Hence, there exist sequences $(\bar s_{j,i})_{i\in\cN}$, for  $j=1,\dots,N,$ such that 
$$
\bar s_{j,i}\in D_{k(j)}  \quad {\rm and } \quad  \bar s_{1,i}<\dots<\bar s_{N,i}, \quad \text{for all } i\in\cN, \quad {\rm and } \quad \lim_{i \to +\infty} \bar s_{j,i}=\bar s_j,
$$
For each $i\in\cN$, set $F_i:=\Big\{(\bar s_{1,i}, \cbf_1),\dots,(\bar s_{N,i},\cbf_N)\Big\}$, so that $F_i$ has the form \eqref{Fform} and hence $\Lambda(F_i)\ne\emptyset$. For each $i\in\cN$, let us select  $(\bar p_{0_i},\bar p_i,\lambda_i)\in \Lambda(F_i)$.  Since $|(\bar p_{0_i},\bar p_i,\lambda_i)|=1$, by possibly taking  a subsequence, we can assume that $(\bar p_{0_i},\bar p_i,\lambda_i)$ converges to a point $(\bar p_0,\bar p,\lambda)$ with $|(\bar p_0,\bar p,\lambda)|=1$.
By the definition of $D_{k(j)} (\subseteq E_{k(j)} )$, passing to the limit as  $i\to+\infty$ one obtains that  $(\bar p_0,\bar p,\lambda)\in\Lambda(F)$. Hence we have proved  that  $\Lambda(F)\ne\emptyset$ as soon as  ${\rm card}(F)<+\infty$ \footnote{Here ${\rm card}(Q)$ denotes the cardinality of the set $Q$}. 
In particular, if  we take a finite family of subsets  $F_1,\dots,F_M \subset D\times  \mathfrak{V}$ with ${\rm card}(F_i)<+\infty$ for all $i=1,\dots,M$, we get
 $
\ds\Lambda(F_1)\cap\dots\cap\Lambda(F_M)=\Lambda \left(\cup_{i=1}^MF_i\right)\ne\emptyset.
 $
  Hence 
   $
  \Big\{\Lambda(F): F\subset  D\times  \mathfrak{V}, \ {\rm card}(F)<+\infty\Big\}$ is a family of compact subsets such that the intersection of each finite subfamily is non-empty. This implies that also the (infinite) intersection of all $\Lambda(F)$ over finite sets $F$ is non-empty.  Therefore
$
 \Lambda( D\times  \mathfrak{V})=\Lambda \left( \bigcup_{{\rm card}(F)<+\infty} F \right)=\bigcap_{{\rm card}(F)<+\infty}\Lambda(F)\ne\emptyset.
$
 This means that  there exists some  covector  $(\bar p_0,\bar p,\lambda)\neq 0$ such that, setting $p_0:=\bar p_0$, $p( \cdot):=\bar p\cdot M(\bar S,\cdot)$, for all time  $s$  in the full-measure set $D$,  one gets 
\bel{rel1}
\begin{split}
{\bf H} \big(&\bar y(s),p_0,p(s),0,\lambda\big) =H\Big(\bar y(s),p_0,p(s),0,\lambda,\bar w^0(s),\bar w(s),\bar\alpha(s)\Big)  \\ 
&= \ds\max_{(w^0,w,a,\zeta)\in W\times A\times \left[-\frac12,\frac12\right]}H\Big(\bar y(s),p_0,p(s),0,\lambda,w^0(1+\zeta), w(1+\zeta),a\Big) \\
&= \max_{\zeta\in  \left[-\frac12,\frac12\right]}(1+\zeta)\, {\bf H}\big(\bar y(s),p_0,p(s),0,\lambda\big),
 \end{split}
\eeq
\bel{rel2}
p(s)\cdot B(\bar y(s)) \le 0, \quad \text{for all } B\in \B^0.
\eeq

The first relation in  \eqref{rel1} coincides with  \eqref{fe3}, while the last one immediately implies \eqref{engine}. 
Finally,  observe that 
 $B\in\B^0$ if and only if $-B\in \B^0$, so that  \eqref{rel2} yields  \eqref{pg000hi}. 
This concludes the proof, since,  in case $\bar y(\bar S)>0$, the strengthened  non-triviality condition  \eqref{strongfe1}
can be obtained as in the proof of the First Order Maximum Principle.

\appendix

\section*{Acknowledgments}
This research is partially supported by the  Padua University grant SID 2018 ``Controllability, stabilizability and infimum gaps for control systems'', prot. BIRD 187147; by the ``National Group for Mathematical Analysis, Probability and their Applications"  (GNAMPA-INdAM) (Italy);  by the European Union under the 7th Framework Programme FP7-PEOPLE-2010-ITN  Grant agreement number 264735-SADCO; by FAPERJ (Brazil) trough the {\em ``Jovem Cientista do Nosso Estado''} Program; by CNPq and CAPES (Brazil) and by the Alexander von Humboldt Foundation (Germany). 

\if{

\section{Appendix}

\subsection{Proof of Lemma \ref{Min=} }\label{ProofMin=}
\vsm

 It suffices  to prove the  following assertion:
\begin{itemize}
\item[(a)] there is  $\delta>0$ such that 
 $
\h(\bar T, \bar x(\bar T), \bar \va(\bar T)\,\leq\, \h(T,  x(T),\va(T))
 $
for all  feasible regular trajectory-control pairs  $(T, x,\va,u,a)$ verifying $(T,  u,a)\in \U({\check x})$ and 
\bel{close1App}
d_\infty\Big((T, x,\va),(\bar T,\bar x,\bar \va)\Big) = |T-\bar T|+\|(\bar x,\bar \va) - (x,\va)\|_{L^{\infty}(\cR)}  \leq \delta 
\eeq
if and only if
\item[(b)] there is  $\delta'>0$ such that 
$
\h((\bar y^0, \bar  y,\bar \beta)(\bar S)) \leq  \h((y^0, y,\beta)(S))
$
 for all feasible space-time trajectory-control pairs $(S, y^0, y,\beta, \varphi^0, \varphi,\alpha)$  satisfying $(S, \varphi^0, \varphi,\alpha)\in\Gamma_+({\check x})$ and  
 \begin{equation}
\label{close2App}
\begin{array}{l}
d^e_\infty\Big(S, y^0, y, \beta),( \bar S, \bar y^0,\bar  y,\bar \beta)\Big)=  \\ [1.5ex]
 \qquad\quad  |S-\bar S|+\|(y^0, y,\beta) - (\bar y^0, \bar y , \bar \beta)\|_{L^{\infty}(\cR)} \leq \delta'\,.
 \end{array}
\end{equation}
\end{itemize}
 Notice, to begin with,   that we can consider the  extension   to $\cR$ of any regular feasible trajectory-control pair $(\tilde T,\tilde x,\tilde \va, \tilde u, \tilde a)$  as a solution of the control system 
$$
\left\{
\begin{array}{l}
\displaystyle\frac{dx}{dt}(t)\,=\, f(t,x(t),a(t))\,b(t) + \sum_{j=1}^{m}g_{j}(t,x(t))\frac{du^j}{dt}(t) \, \quad \mbox{ a.e. } t \in \cR,  
\\
\displaystyle\frac{d\va}{dt}(t)\,=\,  \left|\frac{du}{dt}(t)\right| \, \quad \mbox{ a.e. } t \in \cR,\\
\ds(x,\va)(0)=(\tilde x,\tilde \va)(0)
\end{array}\right.
$$
 associated with the control 
$$
(b,u,a)(t):=\left\{\begin{array}{l} (0,\tilde u(\tilde t_1), \hat a) \qquad t<0,  \\ [1.5ex]
(1,\tilde u(t),\tilde a(t)) \qquad t\in [0,\tilde T],  \\ [1.5ex]
(0,\tilde u(\tilde t_2), \hat a) \qquad t>\tilde T,
\end{array}\right.
$$ 
where $\hat a\in A$ is arbitrary.   Similarly, the  extension to $\cR$ of any feasible extended sense process $(\tilde S,\tilde y^0,\tilde  y,\tilde \beta  ,\tilde\varphi^0,\tilde \varphi)$ can be regarded as a solution of the control system
$$
\left\{
\begin{array}{l}
\displaystyle\frac{d{y^0}}{ds} (s)\,=\, \beta(s)\,\frac{d\varphi^0}{ds}(s)   \ \mbox{ a.e. } s \in [0,S]\,,  \\ [1.5ex]
\displaystyle\frac{dy}{ds} (s)\,=\, f(y^0(s), y(s))\beta(s)\frac{d\varphi^0}{ds}(s)+ \sum_{j=1}^{m}g_{j}(y^0(s), y(s))\frac{d\varphi^j}{ds}(s)   \ \mbox{ a.e. } s \in [0,S]\,, 
\\ [1.5ex]
\displaystyle\frac{d\beta}{ds} (s)\,=\,\left|\frac{d\varphi}{ds}(s)\right|  \ \mbox{ a.e. } s \in [0,S]\,, 
\\ [1.5ex]
(y^0,y,\beta)(0)=(\tilde y^0,\tilde  y,\tilde \beta)(0)
\end{array}\right.
$$
associated  with the control 
$$
(\beta, \varphi^0,\varphi,\alpha)(s):=\left\{\begin{array}{l} (0,\tilde\varphi^0(0)+s,\tilde \varphi(0),\hat a) \qquad\qquad s<0,  \\ [1.5ex]
(1,\tilde\varphi^0,\tilde \varphi,\tilde\alpha)(s) \qquad\qquad\qquad s\in[0,\tilde S],  \\ [1.5ex]
(0,\tilde\varphi^0(\tilde S)+(s-\tilde S),\tilde \varphi(\tilde S),\hat a) \qquad s>\tilde S.
\end{array}\right.
$$ 
 The map ${\mathcal I}$ can be defined on these  $\cR$-extended trajectory-control pairs.   In particular, the assertions regarding ${\mathcal I}$  are still valid for  this extension of ${\mathcal I}$. 
We make the convention that in  the rest of the proof  all trajectory-control pairs are  extended to $\cR$ as described above.
\vsm
Let us prove that  (a) $\Longrightarrow$ (b).  Let $(\bar T,\bar x,\bar \va, \bar u,\bar a)$  be a  regular  local minimizer as in (a), let $(\bar S,\bar y^0, \bar y, \bar \beta, \bar\varphi^0,\bar\varphi,\bar\alpha)$ be the canonical representative in ${\mathcal I}(\bar T,\bar x,\bar \va, \bar u,\bar a)$ and for any  space-time trajectory-control pair $(S,y^0, y, \beta, \varphi^0,\varphi,\alpha)$ with  $(S, \varphi^0,\varphi,\alpha)\in\Gamma_+({\check x})$, let us set $(T, x,\va ,u,a  ):={\mathcal I}^{-1}[(S,{y^0}  , y  , \beta, \varphi^0  ,\varphi,\alpha  )]$. 

\noindent For every $t\in\cR$, set $s:=(\varphi^0)^{-1}(t)=\sigma(t)$ and  $\bar s:=(\bar\varphi^0)^{-1}(t)=\bar\sigma(t)$, so that $|s-\bar s|=|\sigma(t)-\bar\sigma(t)|$ (and $\varphi^0(s)=t=\bar\varphi^0(\bar s)$). By  definition of  ${\mathcal I}$ one has
$$
\begin{array}{l}
|\bar T-T|+ |(\bar x,\bar v)(t)- (x,v)(t)|  =|\bar\varphi^0(\bar S)-\varphi^0(S)|+ |(\bar y,\bar\beta)(\bar s)-(y,\beta)(s)|\le \\ [1.5ex]
|\bar\varphi^0(\bar S)-\bar\varphi^0(S)|+|\bar\varphi^0(S)-\varphi^0(S)|+ |(\bar y,\bar\beta)(\bar s)-(\bar y,\bar\beta)(s)|+|(\bar y,\bar\beta)(s)-(y,\beta)(s)|\le \\ [1.5ex]
L|\bar S-S|+\|\bar\varphi^0-\varphi^0\|_{L^\infty(\cR)}+L|\bar s-s|+\|(\bar y,\bar\beta)-(y,\beta)\|_{L^\infty(\cR)}
  \le \\
  C\,d^e_\infty\Big(S, y^0, y, \beta),( \bar S, \bar y^0,\bar  y,\bar \beta)\Big)+L|\bar s-s|,
\end{array}
$$
where the last inequalities follow from  the Lipschitzianity of  the controls $\bar\varphi^0$, $\bar\varphi$, which implies  that the trajectory  $(\bar y^0,\bar y,\bar \beta)$ is Lipschitz continuous too.  Here and in the sequel  $L>0$, $C>0$  denote generic Lipschitz constants and upper bounds, respectively.

\noindent  Let us set $\bar t:=\bar\varphi^0(s)$, so that $\bar\sigma(\bar t)=s=\sigma(t)$. Then
 $$
  |s-\bar s|=|\sigma(t)-\bar\sigma(t)|=|\bar\sigma(\bar t)-\bar\sigma(t)|\le\omega_{\bar\sigma}(|\bar t-t|)=\omega_{\bar\sigma}(|\bar\varphi^0(s)-\varphi^0(s)|)\le \omega_{\bar\sigma}(\|\bar\varphi^0-\varphi^0\|_{L^\infty(\cR)}),
  $$
where $\omega_{\bar\sigma}$ is the modulus of continuity of $\bar\sigma$, which exists since $\bar\sigma=(\bar\varphi^0)^{-1}$ is absolutely continuous and thus uniformly continuous.  Hence
$$
\begin{array}{l}
d_\infty\left((T, x,\va),(\bar T,\bar x,\bar \va)\right)  \le \\ \qquad \qquad C\,d^e_\infty\left((S, y^0, y, \beta),( \bar S, \bar y^0,\bar  y,\bar \beta)\right)+ 
 L \omega_{\bar\sigma}\left(d^e_\infty\left((S, y^0, y, \beta),( \bar S, \bar y^0,\bar  y,\bar \beta)\right)\right), 
\end{array}
$$
 which implies assertion (b),  provided we choose $\delta'>0$ verifying $C\delta'+\omega_{\bar\sigma}(\delta')<\delta$. Indeed, we obtain
$$
\begin{array}{l}
\h((\bar y^0, \bar  y,\bar\beta)(\bar S))=\h(\bar T,  \bar x(\bar T), \bar\beta(\bar T))\le \\ [1.5ex]
\qquad\qquad\qquad \h(T,  x(T),\va(T))=\h((y^0, y,\beta) (S)),
 \end{array}
$$ 
  for all feasible space-time trajectory-control pair  with $\frac{d\varphi_0}{ds}>0$ a.e. and   satisfying \eqref{close2App}  for such $\delta'$. 

\vsm
Let us now show  that (b) $\Longrightarrow$ (a).  Let $(\bar S,\bar y^0  ,\bar  y ,\bar\beta, \bar \varphi^0  ,\bar \varphi, \bar\alpha )$ be a $L^{\infty}$  local minimizer for problem $(P_e)$   among feasible space-time trajectories with controls in $\Gamma_+({\check x})$, as in (b). 
 Without loss of generality we  assume that   $(\bar S,\bar y^0  ,\bar  y ,\bar\beta, \bar \varphi^0  ,\bar \varphi )$ is canonical.   
For  any regular feasible trajectory-control pair $(T, x ,\va, u ,a )$ let $(S,y^0, y, \beta, \varphi^0  ,\varphi,\alpha)$ be the canonical representative in ${\mathcal I}(T, x,\va , u ,a )$.  

\noindent For every $s\in\cR$ let  $t:=\varphi^0(s)$,  $\bar t:=\bar\varphi^0(s)$.  Set  $\sigma :=(\varphi^0)^{-1}$,  $\bar\sigma:=(\bar\varphi^0)^{-1}$ and $\bar s:=\bar\sigma(t)$.   Notice that the choice of canonical controls implies, for every $s\in\cR$, 
$$
\sigma(t)=s=y^0(s)+\beta(s)=  t+\va(t), \quad \bar\sigma(t)=\bar s=\bar y^0(\bar s)+\bar\beta(\bar s)=   t+\bar\va( t),
$$
so that 
$$
 |\bar s-s|=| \bar\sigma(t)- \sigma(t)|= |\bar\va(t)-\va(t)|.
 $$ 
Then  
$$
\begin{array}{l}
|\bar S-S|+|(\bar y^0, \bar y,\bar\beta)(s)-(y^0,y,\beta)(s)|\le   \\ [1.5ex]
|\bar T+\bar\va(\bar T)-T-\va(T)|+|(\bar y^0, \bar y,\bar\beta)(s)-(\bar y^0, \bar y,\bar\beta)(\bar s)|+|(\bar y^0, \bar y,\bar\beta)(\bar s)-(y^0,y,\beta)(s)|\le  \\ [1.5ex]
|\bar T-T|+|\bar\va(\bar T)- \va(T)|+L|\bar s-s|+|(t, \bar x( t),\bar \va(t))-(t,x(t), \va(t))|\le \\ [1.5ex]
 \,d_\infty\left((\bar T, \bar x,\bar\va), (T,x,\va)\right)+(1+L)\|\bar \va-\va\|_{L^\infty(\cR)}, 
\end{array}
$$
  where $L$ denotes the Lipschitz constant of $(\bar y^0, \bar y,\bar\beta)$  and we get
$$
d^e_\infty\Big(S, y^0, y, \beta),( \bar S, \bar y^0,\bar  y,\bar \beta)\Big)\le C\, d_\infty\left((\bar T, \bar x,\bar\va), (T,x,\va)\right).
$$
for some $C>0$.  Therefore choosing any $\delta >0$
satisfying  $C\delta<\delta'$, we have 
$$
\begin{array}{l}
\h(\bar T,  \bar x(\bar T), \bar\beta(\bar T))=\h((\bar y^0, \bar  y,\bar\beta)(\bar S))\le \\ [1.5ex]
\qquad\qquad\qquad \h((y^0, y,\beta) (S))=\h(T,  x(T),\va(T)),
 \end{array}
$$ 
  for all regular feasible trajectory-control pairs   satisfying \eqref{close1App}.  This confirms (a).  \qed

}\fi

\bibliographystyle{siamplain}
\bibliography{impulsive}

\end{document}